\documentclass[11pt,reqno]{amsart}
\usepackage{amssymb,amsmath,amsfonts,amsthm,enumerate,stmaryrd,upgreek}
\usepackage{mathrsfs}
\usepackage{graphicx}

\usepackage[left=.75 in, right=.75 in,top=.75 in, bottom=.75 in]{geometry}
\usepackage{nameref,hyperref,cleveref}
\numberwithin{equation}{section}

\usepackage{color}

\providecommand{\abs}[1]{\left\vert#1\right\vert}
\providecommand{\norm}[1]{\left\Vert#1\right\Vert}

\providecommand{\ip}[1]{\left(#1 \right)}

\def\nab{\nabla}

\def\hal{\frac{1}{2}}
\def\ep{\varepsilon}

\def\ls{\lesssim}
\def\p{\partial}

 \def\bcr{\begin{color}{red}}
\def\bcb{\begin{color}{blue}}
\def\bcc{\begin{color}{magenta}}
\def\ec{\end{color}}

\providecommand{\abs}[1]{\left\vert#1\right\vert}
\providecommand{\norm}[1]{\left\Vert#1\right\Vert}

\def\hal{\frac{1}{2}}
\def\ep{\varepsilon}

\def\ls{\lesssim}

\def\nab{\nabla}

\def\p{\partial}

\newcommand{\m}{\mathrm}
\def\pint{\m{P}_{\m{int}}}
\def\pfint{\mathfrak{P}_{\m{int}}}
\def\hint{\upeta_{\m{int}}}
\def\pext{\m{P}_{\m{ext}}}

\def\h{\mathcal{H}}

\def\L{\mathcal{L}}

\def\N{\mathbb{N}}

\def\R{\mathbb{R}}

\def\acl{AC_{\m{loc}}}

\def\st{\;\vert\;}

\def\XXint#1#2#3{{\setbox0=\hbox{$#1{#2#3}{\int}$ }
\vcenter{\hbox{$#2#3$ }}\kern-.6\wd0}}

\DeclareMathOperator{\spn}{span}
\DeclareMathOperator{\ran}{ran}

\newtheorem{lem}{Lemma}[section]

\newtheorem{prop}[lem]{Proposition}
\newtheorem{thm}[lem]{Theorem}

\newtheorem{dfn}[lem]{Definition}

\title[Hydrostatic bubbles]{Hydrostatic bubbles of compressible fluid in an incompressible fluid}

\author{Juhi Jang}
\address{
Department of Mathematics \\
University of Southern California\\
Los Angeles, CA 90089, USA \\
and 
Korea Institute for Advanced Study \\
Seoul, Korea
}
\email[J. Jang]{juhijang@usc.edu}
\thanks{J. Jang was supported by NSF Grants (DMS \#2306910 and \#2605589)}

\author{Ian Tice}
\address{
Department of Mathematical Sciences\\
Carnegie Mellon University\\
Pittsburgh, PA 15213, USA
}
\email[I. Tice]{iantice@andrew.cmu.edu}
\thanks{I. Tice was supported by an NSF Grant (DMS \#2508400)}



\begin{document}

\begin{abstract} 
We study stationary configurations of compressible barotropic fluids lying inside an incompressible fluid and acted upon by a constant gravitational field.  Without gravity, it is a simple matter to construct solutions consisting of perfectly spherical bubbles of compressible fluid, but with gravity the stratification of the compressible fluid's density prevents the existence of solutions with such simple geometries.  We construct smooth solutions with nontrivial gravity and nearly spherical, axisymmetric geometry by means of a bifurcation argument in an appropriate weighted Sobolev space, defined in terms of the Legendre operator.  Along the way, we develop a number of essential linear and nonlinear functional analytic tools for this scale of spaces, which are of independent interest. 
\end{abstract}

\maketitle

\section{Introduction }
  
\subsection{Overview }

In this paper we study hydrostatic bubbles of compressible fluid within an incompressible fluid in the presence of a constant gravitational field.     The two fluids together fill a given and fixed open set $\Omega_{\m{tot}} \subseteq \R^3$, which we assume decomposes into the disjoint union 
\begin{equation}
\Omega_{\m{tot}} = \Omega_{\m{ext}} \cup \overline{\Omega_{\m{int}}} = \Omega_{\m{ext}} \cup \Omega_{\m{int}} \cup \Sigma,    
\end{equation}
with $\Omega_{\m{ext}}$ an open set occupied by the incompressible fluid and  $\Omega_{\m{int}}$ an open, bounded, and simply connected set occupied by the compressible fluid.  In addition, we will assume that $\overline{\Omega_{\m{int}}}$ is a smooth manifold with boundary, so $\Sigma = \p \Omega_{\m{int}} \subset \R^3$  is then a smooth, compact two-dimensional submanifold.  This is a free boundary problem, as the sets $\Omega_{\m{int}}$ and $\Omega_{\m{ext}}$ are not specified a priori and must be constructed.  In doing so we will solve for $\overline{\Omega_{\m{int}}} \subset \Omega_{\m{tot}}$ and then define $\Omega_{\m{ext}} = \Omega_{\m{tot}} \backslash \overline{\Omega_{\m{int}}}$, which justifies our notation indicating the compressible bubble is the interior and the incompressible fluid is the exterior.  We emphasize that there are many possibilities for the set $\Omega_{\m{tot}}$: it can be taken to be all of $\R^3$, a bounded open set modeling a container with rigid walls, a horizontally infinite layer, or something more complicated.  Regardless, we will always assume the coordinate system in $\R^3$ is chosen so that $0 \in \Omega_{\m{int}} \subset  \Omega_{\m{tot}} $.  It will be useful to define 
\begin{equation}\label{R_max_def}
 R^{\m{max}} = \sup\{ s >0 \st s \mathbb{B}^3 = B(0,s) \subset \Omega_{\m{tot}}   \} \in (0,\infty],
\end{equation}
where $\mathbb{B}^3= B(0,1)$ is the open unit ball in $\R^3$.

Our interest in this paper is hydrostatic solutions acted upon by a constant gravitational field $-g e_3 = (0,0,-g) \in \R^3$, which we take to mean that both fluids have trivial velocity fields.   In turn, this means that viscous effects are irrelevant and that the fluid configurations are determined solely by their density and pressure profiles, together with a jump condition involving the curvature of $\Sigma$, known as the Laplace-Young equation:
\begin{equation}\label{hydrostatic_eqns}
\begin{cases}
\nab \pext = - \rho_{\m{ext}} g e_3 &\text{in } \Omega_{\m{ext}} \\ 
\nab [\pint \circ \rho_{\m{int}}] =  - \rho_{\m{int}} g e_3 &\text{in } \Omega_{\m{int}}  \\
\sigma \mathscr{K}_{\Sigma} = \pint\circ \rho_{\m{int}} - \pext &\text{on } \Sigma.
\end{cases}
\end{equation}
The various terms appearing in \eqref{hydrostatic_eqns} are as follows.  The parameter $g \in \R$ is the gravitational constant; in principle we could consider only $g \ge 0$ to give the gravitational field a preferred downward direction, but the theory we develop in this paper works more generally.  The incompressible fluid's density is the given constant $\rho_{\m{ext}}  \in \R_+ := (0,\infty)$, and its pressure is the unknown $\pext : \Omega_{\m{ext}} \to \R$.  The compressible fluid's density is the unknown function $\rho_{\m{int}} : \Omega_{\m{int}} \to \R_+$, which we assume determines its pressure through a barotropic equation of state.  More precisely, we assume we are given  a smooth function  $\pint : \R_+ \to \R_+$ satisfying $\pint' >0$ and $\lim_{t \to 0} \pint(t) =0$; then the pressure in the compressible fluid is $\pint \circ \rho_{\m{int}} : \Omega_{\m{int}} \to \R_+$.   Finally, the term $\sigma \in \R_+$ is the coefficient of surface tension, and $\mathscr{K}_{\Sigma}$ is the total curvature (twice the mean curvature) of the surface $\Sigma$, normalized so that $\mathscr{K}_{\mathbb{S}^2} = 2$ for $\mathbb{S}^2 =\p \mathbb{B}^3 \subset \R^3$ the  unit sphere.

It turns out to be trivial to solve for $\pext$ satisfying the first equation in \eqref{hydrostatic_eqns} in all of $\R^3$ and relatively easy to solve for $\rho_{\m{int}}$ satisfying the second equation in a slab-like set of the form $\R^2 \times [-r,r]$ for $0 < r < R^{\m{max}}$.  With these in hand, our task will then reduce   to constructing $\overline{\Omega_{\m{int}}} \subset [ \R^2 \times [-r,r]] \cap \Omega_{\m{tot}}$ satisfying the third equation in \eqref{hydrostatic_eqns}.  Before further discussing this task, we justify these solvability claims.

The first equation in \eqref{hydrostatic_eqns} is solved by $\pext : \R^3 \to \R$ given by 
\begin{equation}\label{pext_def}
 \pext(x) = \pext^\ast - g \rho_{\m{ext}} x_3 
\end{equation}
for a given reference pressure $\pext^\ast \in \R_+$.  To solve for $\rho_{\m{int}}$ we must first introduce some notation.  It will be convenient to write $\pint^{\m{max}} = \sup_{t > 0} \pint(t) \in (0, \infty] $ so that $\pint$ is a smooth, increasing diffeomorphism from $\R_+$ to $\pint(\R_+) = (0,\pint^{\m{max}})$.  Next, we define the smooth enthalpy function $\hint : \R_+ \to\R$ via 
\begin{equation}
 \hint(z) =  \int_{\rho_{\m{ext}}}^z t^{-1} \pint'(t) dt.
\end{equation}
Note that $\hint' >0$, and so $\hint$ is a smooth, increasing diffeomorphism from $\R_+$ to  $\hint(\R_+)  = (\hint^{\m{\min}}, \hint^{\m{max}})$, where 
\begin{equation}
-\infty \le  \hint^{\m{min}} = - \int_{0}^{\rho_{\m{ext}}} t^{-1} \pint'(t) dt <0  < \int_{\rho_{\m{ext}}}^\infty t^{-1} \pint'(t) dt  =\hint^{\m{max}} \le \infty. 
\end{equation}
Note also that we have chosen the normalization 
\begin{equation}\label{enthalpy_normalization}
\hint(\rho_{\m{ext}}) = 0. 
\end{equation}
The point of introducing $\hint$ is that the equation 
\begin{equation}
 \nab[\pint\circ \rho_{\m{int}}] = - g \rho_{\m{int}} e_3
\end{equation}
reduces to 
\begin{equation}
  \hint \circ \rho_{\m{int}}(x)  =\alpha -gx_3
\end{equation}
for some arbitrary $\alpha \in \R$.  We can then solve for $\rho_{\m{int}} : \R^2 \times [-r,r] \to \R_+$ via 
\begin{equation}\label{rhoint_def}
 \rho_{\m{int}}(x) = \hint^{-1}(\alpha - g x_3),
\end{equation}
provided that  $\alpha \in \hint(\R_+)$, $0 < r < R^{\m{max}}$, and 
\begin{equation}\label{rhoint_def_compat}
 \alpha - g [-r,r] \subseteq \hint(\R_+).
\end{equation}
This suggests that we introduce  the smooth, increasing diffeomorphism $\pfint : (\hint^{\m{min}}, \hint^{\m{max}}) \to (0,\pint^{\m{max}})$ via $\pfint = \pint \circ \hint^{-1}$.  Indeed, with this notation in hand we then find that the pressure in the compressible fluid is 
\begin{equation}\label{pint_pfint_def}
 \pint \circ \rho_{\m{int}}(x) = \pfint(\alpha -gx_3).
\end{equation}

Before proceeding, we record a couple useful identities involving $\pfint$ that will be used throughout the paper.  By construction, we have that $\hint'(z) = \pint'(z)/z$.  Thus, 
\begin{equation}
 [\hint^{-1}]'(z) = \frac{1}{\hint'\circ \hint^{-1}(z) } = \frac{\hint^{-1}(z)}{\pint'\circ \hint^{-1}(z)}   \text{ for } z \in \hint(\R_+). 
\end{equation}
In turn, this implies that 
\begin{multline}
 \pfint'(z) = \pint' \circ \hint^{-1}(z)  [\hint^{-1}]'(z) =  \hint^{-1}(z), \; 
\pfint''(z) = \frac{\hint^{-1}(z)}{\pint'\circ \hint^{-1}(z)},  \\
\text{ and } \pfint'''(z) = \frac{\hint^{-1}(z) \pint'\circ \hint^{-1}(z) - (\hint^{-1}(z))^2 \pint''\circ \hint^{-1}(z)}{(\pint'\circ \hint^{-1}(z))^3}
 \text{ for } z \in \hint(\R_+). 
\end{multline}
In particular, this and the normalization \eqref{enthalpy_normalization} require that
\begin{equation}\label{bifurcation_thm_6}
\pfint'(0) = \hint^{-1}(0) = \rho_{\m{ext}}, \;
\pfint''(0) = \frac{\hint^{-1}(0)}{\pint' \circ \hint^{-1}(0)} =  \frac{\rho_{\m{ext}}}{\pint'(\rho_{\m{ext}})} >0,
\end{equation}
and 
\begin{equation}\label{bifurcation_thm_6_5}
 \pfint'''(0) = \frac{\rho_{\m{ext}} \pint'(\rho_{\m{ext}}) - \rho_{\m{ext}}^2 \pint''(\rho_{\m{ext}})}{(\pint'(\rho_{\m{ext}}))^3}. 
\end{equation}

\subsection{Spherical bubbles}\label{subsec:1.2}

When we neglect gravity by setting $g =0$ in \eqref{hydrostatic_eqns}, then \eqref{pext_def}, \eqref{rhoint_def}, and \eqref{pint_pfint_def} reduce to the constants
\begin{equation}\label{g0_soln}
 \pext = \pext^\ast, \ \rho_{\m{int}} = \hint^{-1}(\alpha), \text{ and } \pint \circ \rho_{\m{int}} = \pfint(\alpha)  
\end{equation}
for arbitrary $\alpha \in \hint(\R_+)$, which are defined throughout $\R^3$. In particular, this means that the right-hand side of the third equation in \eqref{hydrostatic_eqns} is constant. By Alexandrov's soap bubble theorem \cite{A_1958}, the only compact, connected, embedded hypersurface with constant mean curvature is a sphere, and hence we must choose $\Omega_{\m{int}} = R_\alpha \mathbb{B}^3 = B(0,R_\alpha)$.  As the total curvature of a sphere is twice the reciprocal of its radius,  for such a ball we have a solution if and only if the radius  $R_\alpha \in \R_+$ satisfies
\begin{equation}
 \frac{2\sigma}{R_\alpha} = \pfint(\alpha) - \pext^\ast.
\end{equation}
Now, we can solve this for $R_\alpha$ precisely when $\pext^\ast \in (0, \pint^{\m{max}})$, in which case we have a one-parameter family of solutions with 
\begin{equation}\label{Ralpha_def}
 R_\alpha = \frac{2\sigma}{\pfint(\alpha) - \pext^\ast} \text{ for all } \alpha \in (\pfint^{-1}(\pext^\ast)  ,\hint^{\m{max}}).
\end{equation}
The final condition to check in order to arrive at a proper solution to \eqref{hydrostatic_eqns} is that $R_\alpha <  R^{\m{max}}$, where the latter was defined in \eqref{R_max_def}, so that we can indeed guarantee the inclusion $\overline{\Omega_{\m{int}}} \subset \Omega_{\m{tot}}$.  To summarize, when $g=0$ we have a perfectly spherical bubble solution to \eqref{hydrostatic_eqns} with $\Omega_{\m{int}} = R_\alpha \mathbb{B}^3$ and \eqref{g0_soln} precisely when 
\begin{equation}\label{g0_alpha_range}
 \frac{2\sigma}{R^{\m{max}}} + \pext^\ast \in (0,\pint^{\m{max}}) \text{ and } \alpha \in (\pfint^{-1}(\pext^\ast + 2\sigma /R^{\m{max}} )  ,\hint^{\m{max}}).
\end{equation}
As the parameter $\alpha$ increases, the bubble radius decreases and the pressure inside the bubble increases. 

The simple story above changes quite a bit when gravity is nontrivial, i.e. $g \neq 0$.  In this case the vertical stratification of $\pext(x)$ and $\pint\circ \rho_{\m{int}}(x)$ implies that spherical bubble solutions are no longer available.  

\subsection{Main result}

Our main result in this paper establishes that under a minor strengthening of the assumption \eqref{g0_alpha_range}, nearly spherical bubble solutions persist for all choices of $g$ in an open neighborhood of $0$.  The strengthening of \eqref{g0_alpha_range} we need is that 
\begin{equation}\label{necessary_condition}
 \frac{2\sigma}{R^{\m{max}}} + \pext^\ast < \pint(\rho_{\m{ext}}), \text{ or equivalently }  \pfint^{-1}\left(\frac{2\sigma}{R^{\m{max}}} + \pext^\ast\right) < 0 < \hint^{\m{max}},
\end{equation}
where the equivalence follows from the normalization \eqref{enthalpy_normalization}.  We can now state our result after noting that we define the norm on the Banach space $C^m_b([-1,1])$ for $m \in \N$ to be 
\begin{equation}
 \norm{f}_{C^m_b} = \max_{0\le k \le m} \sup_{\abs{\zeta} \le 1} \abs{D^k f(\zeta)}.
\end{equation}

\begin{thm}\label{main_thm}
Assume \eqref{necessary_condition} holds.  Then there exist an open, symmetric interval $0 \in G \subseteq \R$ and functions $A :G  \to \R$ and $\varphi : G \to C^\infty([-1,1])$ such that the following hold.
\begin{enumerate}
 \item $A(0) =0$, $A(g) = A(-g)$, and $\pfint^{-1}\left(\frac{2\sigma}{R^{\m{max}}} + \pext^\ast\right) < A(g) < \hint^{\m{max}}$ for all $g \in G$, and $A$ is smooth.  Moreover, as $g \to 0$ we have the asymptotic expansion
\begin{equation}
 A(g) =  - g^2 \frac{\pfint'''(0)R_0^2 }{ 10\pfint''(0) } + O(g^4).
\end{equation}

 \item $\varphi(0)=0$, and for any $m \in \N$ the induced map $\varphi : G \to C^m_b([-1,1])$ is smooth.  Moreover, for $g \in G$ we have that $\varphi(-g)(\zeta) = \varphi(g)(-\zeta)$, and as $g \to 0$ we have the asymptotic expansion 
\begin{equation}
\varphi(g)(\zeta) = g^2 \frac{\pfint''(0) R_0^4}{8\sigma}(\zeta^2 - 1)  - g^3 \frac{\pfint'''(0) R_0^5}{300\sigma} (5\zeta^3 - 3 \zeta) + O(g^4).
\end{equation}

 \item For every $g \in G$ the map $\Lambda_g : \overline{\mathbb{B}^3} \to \R^3$ defined by $\Lambda_g(x) = (R_{A(g)} + \varphi(g)(x_3)) x$ is a smooth diffeomorphism onto its image.  Moreover, if we set
\begin{equation}
 \Omega_{\m{int},g} = \Lambda_g(\mathbb{B}^3) \text{ and } \Sigma_g = \Lambda_g(\mathbb{S}^2) 
\end{equation}
then $\overline{\Omega_{\m{int},g}}$ is a smooth, compact, simply connected manifold with boundary, $\Sigma_g = \p \overline{\Omega_{\m{int},g}}$ is a smooth, compact, simply connected manifold, and $\overline{\Omega_{\m{int},g}} \subset \Omega_{\m{tot}},$ which then allows us to define the open set $\Omega_{\m{ext},g} = \Omega_{\m{tot}} \backslash \overline{\Omega_{\m{int},g}}$.  In particular, $\Omega_{\m{int},0} = R_0 \mathbb{B}^3$ and $\Sigma_0 = R_0 \mathbb{S}^2$ with $R_0 = 2\sigma / (\pint(\rho_{\m{ext}}) - \pext^\ast )$.

 \item For $g \in G$ we have that $\overline{\Omega_{\m{int},g}} \subset \R^2 \times [-r,r]$, where $R_0 < r < R^{\m{max}}$ is  such that \eqref{rhoint_def_compat} holds for $g \in G$ and $\alpha = A(g)$, and hence $\rho_{\m{int},g} : \overline{\Omega_{\m{int},g}} \to \R_+$ defined by $\rho_{\m{int},g}(x)  = \hint^{-1}(A(g) - gx_3)$ is well-defined and smooth.

 \item For $g \in G$ we have that 
\begin{equation}
\begin{cases}
\nab \pext = - \rho_{\m{ext}} g e_3 &\text{in } \Omega_{\m{ext},g} \\ 
\nab [\pint \circ \rho_{\m{int},g}] =  - \rho_{\m{int},g} g e_3 &\text{in } \Omega_{\m{int},g}  \\
\sigma \mathscr{K}_{\Sigma_g} = \pint\circ \rho_{\m{int},g} - \pext &\text{on } \Sigma_g.
\end{cases}
\end{equation}
\end{enumerate}

\end{thm}

To the best of our knowledge, Theorem \ref{main_thm} provides the first rigorous construction of non-spherical  bubbles of compressible fluid within an incompressible fluid in a uniform gravitational field.   From a mathematical perspective, the theorem is essentially a well-posedness result with the gravitational constant playing the role of data.  Indeed, the existence assertion in the theorem is clear, and the smooth dependence of the maps $A$ and $\varphi$ on $g$ can be viewed as smooth dependence of the solution on the data.  Theorem \ref{main_thm} does not state a uniqueness result, but its proof is directly based on Theorem \ref{bifurcation_thm}, which does contain a local uniqueness result of sorts.  This is difficult to state precisely at this point in the paper, but roughly speaking it says that there exists an open neighborhood of $0$ in an appropriate Hilbert space such that the only solutions in the open set are either spherical bubbles with trivial gravity or else the solutions with gravity from the theorem.

The theorem and the identities \eqref{bifurcation_thm_6} and \eqref{bifurcation_thm_6_5} show that if we write $\Lambda_g(x) = \lambda_g(x_3) x$, then
\begin{multline}\label{intro_expansion}
\lambda_g(x_3)  =  R_0  + g^2 \frac{R_0^4 \rho_{\m{ext}}}{\sigma \pint'(\rho_{\m{ext}})} \left( \frac{\pint'(\rho_{\m{ext}}) - \rho_{\m{ext}} \pint''(\rho_{\m{ext}})}{20 \pint'(\rho_{\m{ext}})} + \frac{1}{8}(x_3^2 - 1) \right) \\
- g^3 \frac{R_0^5 \rho_{\m{ext}}}{300\sigma (\pint'(\rho_{\m{ext}}))^2  } \left( \frac{\pint'(\rho_{\m{ext}}) - \rho_{\m{ext}} \pint''(\rho_{\m{ext}})}{\pint'(\rho_{\m{ext}})} \right) (5x_3^3 - 3 x_3)  + O(g^4) \text{ as } g\to 0.
\end{multline}
Interestingly, this shows that the leading order correction to the spherical geometry is even in $x_3$ and that one has to go to the $O(g^3)$ term to see an odd contribution.  Additionally, since $R_0, \rho_{\m{ext}}, \pint'(\rho_{\m{ext}})  >0$  and since 
\begin{equation}
\frac{\pint'(\rho_{\m{ext}}) - \rho_{\m{ext}} \pint''(\rho_{\m{ext}})}{ \pint'(\rho_{\m{ext}})} =  - \frac{\rho_{\m{ext}} \hint''(\rho_{\m{ext}})}{ \hint'(\rho_{\m{ext}})} =: Q,
\end{equation}
the shape of the compressible fluid's enthalpy function $\hint$ at the outer incompressible fluid density $\rho_{\m{ext}}$ plays an important role in two terms appearing in the expansion.  If $Q< 0$ then the $O(g^2)$ contribution is negative and contractive.  On the other hand, if $Q>0$ is large enough then the $O(g^2)$ term is positive and expansive; if $Q>0$ is not too large then the $O(g^2)$ term changes sign as a function of $x_3$, with the $O(g^2)$ term positive near the poles and negative near the equator.  We can also  write
\begin{equation}
 \lambda_g(\pm 1) = R_0 + Q[ g^2 C_2 \mp g^3 C_3] +O(g^4) \text{ and } \lambda_g(0) = R_0 + g^2(Q C_2 -C_2')  +O(g^4)\text{ for } C_2,C_2',C_3 >0
\end{equation}
and note that as $g \to 0$ we have $g^2 C_2 \mp g^3 C_3 >0$.  This means that near the poles the sign of $Q$ determines if the bubble is expanding ($Q>0$) or contracting ($Q<0$), though for $g >0$ we always have that $g^2 C_2 + g^3 C_3  > g^2 C_2 - g^3 C_3$ so the magnitude of this deformation is always greater at the south pole (reversing to the north pole if $g < 0$).  Moreover, $\lambda_g(\pm 1) - \lambda_g(0) = g^2 C_2' \mp g^3 Q C_3 +O(g^4) >0$ as $g \to 0$, so $\lambda_g$ is always larger near the poles than near the equator.  As an explicit example, if $\pint$ is a standard polytropic law, say $\pint(\rho) = K \rho^\gamma$ for $K,\gamma >0$, then $Q=2-\gamma >0$ if and only if $0 < \gamma < 2$, in which case the $O(g^2) $ term does change sign.  On the other hand, if $\gamma=2$, then $Q=0$ and $\lambda_g (x_3)= R_0 + g^2  C_2' (x_3^2-1) + O(g^4)$, which implies that to leading order the bubble contracts everywhere except at the poles, with the shortest distance from the center occurring at the equator. 

It is instructive to compare the geometric features of our compressible bubbles with those identified in previous studies.  While the dynamic evolution of spherical bubbles is a classical subject, often modeled via the Rayleigh-Plesset equation or its thermally damped variants (see, for instance, Prosperetti \cite{Pro_1991}), the rigorous mathematical construction of steady bubbles has mostly concentrated on incompressible fluids.  In the viscous regime, the existence of steady falling drops in Navier-Stokes fluids was pioneered by Bemelmans \cite{B_1981} and Solonnikov \cite{S_1999}, with recent advances by Eiter, Kyed, and Shibata \cite{EKS_2023}. In the inviscid regime, formal asymptotic predictions of oblate and prolate bubble shapes have appeared in the physics literature in various analyses: see, for instance, Moore \cite{M_1959}, Harper \cite{H_1972}, and Pozrikidis \cite{P_1989}.  These symmetric leading-order deformations were recently validated mathematically by Meyer, Niebel, and Seis \cite{MNS_2025}, who rigorously constructed traveling vortex bubbles by applying the bifurcation theorem of Crandall and Rabinowitz \cite{CR_1971} on the $2$-sphere, complementing recent work by Lai and Weinstein \cite{LaiW_2025} on asymmetric deformations of perturbed incompressible spherical bubbles.   In contrast to these incompressible models, in which boundary deformation is driven by a flow, our hydrostatic model shows that the geometry of a compressible bubble is determined by the pressure law $\pint$, with the parameter $Q$ playing a key role in the resulting bubble shape.

 Our proof of Theorem \ref{main_thm} shows that the bubbles with nontrivial gravity bifurcate from the curve of spherical solutions without gravity at $\alpha =0$, at which point the constant densities and pressures are related via $\rho_{\m{int}} = \rho_{\m{ext}}$ and $\pext^\ast < \pint(\rho_{\m{int}})$ in light of \eqref{enthalpy_normalization} and \eqref{necessary_condition}.  More precisely, the proof is based on Crandall-Rabinowitz bifurcation theory, applied in a special scale of weighted Sobolev spaces, and with the gravitational constant $g$ parameterizing the bifurcating branch.  A key contribution of this work is the identification of a scale of function spaces and the derivation of their nonlinear functional analytic properties needed to apply Crandall-Rabinowitz.  The bubbles we construct are axisymmetric, and the  Fr\'{e}chet derivative of the curvature operator of the unknown surface at the given sphere leads to the key elliptic operator $L-2I $, where $L$ is the Legendre differential operator defined by 
\begin{equation}
L u(\zeta) = - D ((1-\zeta^2) D u(\zeta)) \text{ for }  \zeta \in [-1,1], 
\end{equation}
which is degenerate at the boundary points (corresponding to the north and south poles).  This degeneracy suggests using some form of weighted Sobolev spaces; for instance, weighted spaces provided the appropriate functional framework to study the so-called physical vacuum in compressible gases \cite{JM_2015}.

The Legendre differential operator $L$ turns out to be best understood in weighted Sobolev spaces $\h^k$ (see \eqref{def:Hk} for the precise definition) characterized in terms of the Legendre polynomials, which are eigenfunctions of $L$.  In particular, for all $k\in \mathbb N$ we have that $ L:\h^{k+2}\to \h^k$  is a bounded linear operator (see Theorem \ref{hk_basics}).  Using a Hardy-type inequality (see Proposition \ref{weighted_basic}) and embedding results, we prove that $\h^k$ is in fact an algebra for all $k\ge 2$ (see Theorem \ref{hs_algebra}). We also show that composition with smooth functions induces smooth maps on $\h^k$ when $k\ge 2$ (see Theorem \ref{hk_comp_Cm}).  

With this functional framework in hand, we can address some of the subtleties of the bifurcation analysis.  The kernel of the linearized curvature operator $L-2I$ is generated by the first Legendre polynomial $p_1(\zeta) =\sqrt{3/2}  \zeta$, which corresponds to the vertical translation invariance of the gravity-free ($g=0$) system.  By restricting our analysis to an orthogonal subspace of $\h^k$, we effectively quotient out this translational symmetry, allowing us to identify the bifurcating kernel and co-kernel pair that make the Crandall-Rabinowitz argument work.  It is also worth noting that the degeneracy of $L$ at the poles is a coordinate singularity arising from the restriction of a geometric problem on the unit sphere $\mathbb{S}^2$ to an axisymmetric configuration parameterized by $\zeta \in [-1,1]$.  As such, one could alternatively formulate the free boundary problem on $\mathbb{S}^2$, where the linearized curvature operator assumes the non-degenerate, but more geometrically subtle, form $-\Delta_{\mathbb{S}^2} - 2I$.  While this geometric perspective would avoid the coordinate singularity at the linear level, executing a rigorous bifurcation argument in axisymmetric Sobolev spaces defined on $\mathbb{S}^2$ still requires some functional analytic care (see, for instance, Section 3 of \cite{MNS_2025}).  By developing the theory natively in one dimension, our scale of spaces $\h^k$ provides a self-contained functional calculus that handles general nonlinearities without relying on special geometric structure or lifts to $\mathbb{S}^2$. We believe that these spaces and their properties will be useful in the study of other nonlinear problems involving the Legendre operator, especially in settings where a spherical geometric lift is either unavailable or unnatural.

While we utilize a uniform gravitational field to break the spherical symmetry of the bubble, similar perturbative strategies can be employed to study the effects of rotation. A prominent related free boundary problem stems from compressible gaseous stars governed by the gravitational Euler-Poisson system, for which the non-rotating hydrostatic equilibria are given by the so-called Lane-Emden equation. They are spherically symmetric solutions with trivial velocity, and their densities decrease to zero from the center to the boundary. An important and physically relevant problem is then to construct rotating star solutions with a given angular velocity. Much like our gravitational bubble problem, slowly rotating axisymmetric star solutions were recently obtained by Jang and Makino \cite{JM_2017, JM_2019} and Strauss and Wu \cite{SW_2017}, who perturbed the Lane-Emden solutions using the angular velocity as a bifurcation parameter. This structural parallel suggests an interesting question related to the present work but beyond the scope of this paper: do rotating compressible bubble solutions exist, and what role does gravity play?  In the incompressible regime, progress in this direction has recently been made by Lai and Weinstein \cite{LaiW_2025_2}, who rigorously constructed non-spherically symmetric equilibrium bubbles in a steadily rotating fluid.

\section{Weighted Sobolev spaces}

In this section we introduce various weighted Sobolev spaces on the interval $(-1,1)$ and establish basic functional analytic properties used to prove our main result.  The first subsection concerns uniformly weighted spaces and embeddings via Hardy-type inequalities. In the second subsection, we study a scale of weighted spaces associated with the Legendre differential operator.

\subsection{Uniformly weighted spaces}

We begin with a definition.

\begin{dfn}
Let $k \in \N$ and $\delta \in [-1,\infty)$.  We define the (real) Hilbert space $H^k_{\delta}((-1,1))$ to be the closure of $C^\infty([-1,1])$ with respect to the norm $\norm{\cdot}_{H^k_\delta}$ coming from the inner-product $\ip{\cdot,\cdot}_{H^k_\delta}$, defined as follows.  If $\delta >-1$  we define 
\begin{equation}
 \ip{u,v}_{H^k_\delta} = \sum_{j=0}^k \int_{-1}^1 (1-\zeta^2)^\delta D^j u(\zeta) D^j v(\zeta)  d\zeta,
\end{equation}
while if $\delta=-1$ we define  
\begin{equation}
 \ip{u,v}_{H^k_{-1}} = \sum_{j=0}^k \int_{-1}^1  \frac{1}{(1-\zeta^2) (\log(2/(1-\zeta^2)))^2}  D^j u(\zeta) D^j v(\zeta) d\zeta.
\end{equation}
Note that $H^k_0((-1,1)) = H^k((-1,1))$ is the usual $L^2$-based Sobolev space of order $k$, and that in general we have the inclusion $H^k_\delta((-1,1)) \subseteq H^k_{\m{loc}}((-1,1))$. 
\end{dfn}

The spaces $H^k_\delta((-1,1))$ employ a common weight (determined by $\delta$) for all of the derivatives up to order $k$, but it turns out that if $u \in H^k_\delta((-1,1))$ then $D^j u$ gets an improved weight for each $0 \le j \le k-1$.  To prove this we first need the following technical result, which is a sort of Hardy-type inequality that is a variant of a result proved in \cite{KK_1967}.

\begin{prop}\label{weighted_basic}
Suppose that $f:[0,1] \to \R$ is a continuous function that is differentiable in $(0,1)$.  Then 
\begin{equation}\label{weighted_basic_00}
 \int_0^1 \frac{\abs{f(x)}^2}{x (\log(2/x))^2} dx \le \frac{16}{\log 2} \int_0^1 x (\abs{f(x)}^2 + \abs{f'(x)}^2) dx,
\end{equation}
and if $\alpha \in \R_+$ then
\begin{equation}\label{weighted_basic_01}
 \int_0^1 x^{\alpha-1} \abs{f(x)}^2 dx \le \max\{\alpha^{-1} 2^{\alpha +4}, \alpha^{-2}(4+\alpha 2^{\alpha+2} )  \}  \int_0^1 x^{\alpha+1} (\abs{f(x)}^2 + \abs{f'(x)}^2) dx.
\end{equation}

\end{prop}
\begin{proof}
We begin with the proof of \eqref{weighted_basic_00}. The result is trivial if the integral on the right side of \eqref{weighted_basic_00} is infinite, so we may assume that it is finite, which in turn means that $f,f' \in L^2((\ep,1))$ for all $0 < \ep < 1$.  Fixing such an $\ep$, we use the fundamental theorem of calculus to compute 
\begin{equation}
 \int_\ep^1 \frac{\abs{f(x)}^2}{x (\log(2/x))^2} dx = \int_\ep^1 \abs{f(x)}^2 \frac{d}{dx}(\log(2/x))^{-1} dx = - \int_\ep^1 2 \frac{f(x) f'(x)}{\log(2/x)}dx + \frac{\abs{f(1)}^2}{\log(2)} - \frac{\abs{f(\ep)}^2}{\log(2/\ep)}.
\end{equation}
We then use Cauchy-Schwarz to bound 
\begin{multline}
 \abs{ \int_\ep^1 2 \frac{f(x) f'(x)}{\log(2/x)}dx } \le 2 \left(  \int_\ep^1  \frac{\abs{f(x)}^2}{x (\log(2/x))^2} dx  \right)^{1/2}  \left(   \int_\ep^1  x\abs{f'(x)}^2  dx  \right)^{1/2} \\
 \le \frac{1}{2}  \int_\ep^1  \frac{\abs{f(x)}^2}{x (\log(2/x))^2} dx  + 2 \int_\ep^1  x\abs{f'(x)}^2  dx.
\end{multline}
Combining these two bounds and absorbing, we learn that 
\begin{equation}
  \int_\ep^1 \frac{\abs{f(x)}^2}{x (\log(2/x))^2} dx  \le 4 \int_\ep^1  x\abs{f'(x)}^2  dx + 2\frac{\abs{f(1)}^2}{\log(2)}  
\end{equation}
for all $0 < \ep < 1$.  Sending $\ep \to 0$ and employing the monotone convergence theorem then shows that 
\begin{equation}
  \int_0^1 \frac{\abs{f(x)}^2}{x (\log(2/x))^2} dx  \le 4 \int_0^1  x\abs{f'(x)}^2  dx + 2\frac{\abs{f(1)}^2}{\log(2)}.  
\end{equation}
On the other hand, standard Sobolev theory shows that 
\begin{equation}
 \abs{f(1)}^2 \le 4\int_{1/2}^1 \abs{f(x)}^2 dx + \int_{1/2}^1 \abs{f'(x)}^2 dx \le 2 \int_0^1 x (4\abs{f(x)}^2 + \abs{f'(x)}^2) dx.
\end{equation}
The estimate \eqref{weighted_basic_00} then follows by combining the previous two bounds and noting that $4+ 4/\log(2) \le 16/\log 2$.

Next, we prove \eqref{weighted_basic_01}, again assuming that its right side is finite.  A minor variant of the argument above shows that for $0 < \ep < 1$ we have the bound 
\begin{equation}
 \alpha \int_\ep^1 x^{\alpha-1} \abs{f(x)}^2 dx \le \frac{\alpha}{2} \int_\ep^1 x^{\alpha-1} \abs{f(x)}^2 dx + \frac{2}{\alpha} \int_\ep^1 x^{\alpha+1} \abs{f'(x)}^2 dx + \abs{f(1)}^2.
\end{equation}
Sending $\ep \to 0$, employing the monotone convergence theorem, and using an absorbing argument then shows that 
\begin{equation}
  \int_0^1 x^{\alpha-1} \abs{f(x)}^2 dx \le   \frac{4}{\alpha^2} \int_0^1 x^{\alpha+1} \abs{f'(x)}^2 dx + \frac{2}{\alpha} \abs{f(1)}^2.
\end{equation}
On the other hand, 
\begin{equation}
 \abs{f(1)}^2 \le 4\int_{1/2}^1 \abs{f(x)}^2 dx + \int_{1/2}^1 \abs{f'(x)}^2 dx \le 2^{\alpha+1} \int_0^1 x^{\alpha+1} (4\abs{f(x)}^2 + \abs{f'(x)}^2) dx.
\end{equation}
\end{proof}

The following result records the essential properties of the $H^k_\delta$ spaces we will need in this paper, including the aforementioned improved weights.

\begin{thm}\label{weighted_space_properties}
The following hold.
\begin{enumerate}
 \item If $k_0,k_1 \in \N$ and $\delta_0,\delta_1 \in [-1,\infty)$ satisfy $k_0 \le k_1$ and $\delta_1 \le \delta_0$, then we have the continuous inclusion $H^{k_1}_{\delta_1}((-1,1)) \hookrightarrow H^{k_0}_{\delta_0}((-1,1))$. 
 \item If $1\le k \in \N$ and $1 \le \delta \in \R$, then we have the continuous inclusion $H^k_\delta((-1,1)) \hookrightarrow H^{k-1}_{\delta-2}((-1,1))$.
 \item Suppose $k \in \N$.  Then we have the continuous inclusions $H^{2k}_{2k}((-1,1)) \hookrightarrow H^k((-1,1))$ as well as $H^{2k+1}_{2k+1}((-1,1)) \hookrightarrow H^k_{-1}((-1,1))$.  In particular, we have the continuous inclusion $H^k_k((-1,1)) \hookrightarrow H^{\lfloor k/2\rfloor}((-1,1))$.
 \item If $2\le k \in \N$, then we have the continuous inclusion $H^{k}_{k}((-1,1)) \hookrightarrow C^{\lfloor k/2\rfloor -1, 1/2}_b((-1,1))$.  In particular, each $u \in H^{k}_{k}((-1,1))$ uniquely extends to $u \in C^{\lfloor k/2\rfloor -1, 1/2}_b([-1,1])$.
\end{enumerate}
\end{thm}
\begin{proof}
The first item follows readily from the definition of the norms.  For the second item, we first consider the case $\delta >1$. By applying \eqref{weighted_basic_01} on half-intervals, we then obtain for each $0 \le j \le k-1$, 
\begin{multline}
 \int_{-1}^1 (1-\zeta^2)^{\delta-2} \abs{ D^j u (\zeta)}^2 d\zeta  \lesssim \int_{-1}^0 (1+\zeta)^{\delta-2} \abs{ D^j u(\zeta) }^2 d\zeta + \int_{0}^1 (1-\zeta)^{\delta-2} \abs{ D^j u(\zeta) }^2 d\zeta \\
 \lesssim \int_{-1}^0 (1+\zeta)^{\delta} ( \abs{ D^j u (\zeta)}^2 +  \abs{D^{j+1} u (\zeta)}^2  )d\zeta + \int_{0}^1 (1-\zeta)^{\delta} (\abs{ D^j u (\zeta)}^2 + \abs{D^{j+1} u(\zeta) }^2 )d\zeta  \\
 \lesssim \int_{-1}^1 (1-\zeta^2)^{\delta} (\abs{ D^j u(\zeta) }^2 + \abs{D^{j+1} u(\zeta)}^2 )d\zeta,
\end{multline}
which implies the claimed result. If $\delta=1$, the result follows analogously by applying \eqref{weighted_basic_00} instead. 
The third item follows by applying the second $k$ times in the even case and $k+1$ times in the odd case.  The fourth item follows from the first and third,  combined with the usual one-dimensional Sobolev embedding.
\end{proof}

\subsection{Weighted spaces generated by the Legendre operator}

We now aim to study a scale of weighted Sobolev spaces in terms of the Legendre differential operator.  We begin with a definition.

\begin{dfn}
We define the following.
\begin{enumerate}
 \item The space $\acl((-1,1))  = \{u : (-1,1) \to \R \st u \text{ is absolutely continuous on each compact }K \subset (-1,1)\}$.  Note that if $u \in \acl((-1,1))$ then $Du$ is defined almost everywhere in $(-1,1)$ and is Lebesgue measurable.   We write  $\acl^0((-1,1)) =\acl((-1,1))$.  Then for $1 \le k \in \N$ we define $\acl^k((-1,1)) = \{u \in \acl^{k-1}((-1,1)) \st D^k u \in \acl((-1,1)) \}$.  Then $u \in \acl^k((-1,1))$ implies that $D^j u$ is defined almost everywhere in $(-1,1)$ and Lebesgue measurable for $0 \le j \le k+1$.

 \item Given $1\le k \in \N$, we define the inner-product space $\h^k$ via
\begin{equation}\label{def:Hk}
 \h^k = \{u \in \acl^{k-1}((-1,1)) \st \norm{u}_{\h^k} < \infty \},
\end{equation}
where
\begin{equation}\label{def:Hk_norm}
 \norm{u}_{\h^k}^2 = \sum_{j=0}^k \int_{-1}^1 (1-\zeta^2)^j \abs{D^j u(\zeta)}^2 d\zeta.
\end{equation}
The inner-product on $\h^k$ is the obvious one associated to this square-norm.  When $k=0$ we set  $\h^0 = L^2((-1,1))$ with its standard inner-product.

 \item For $n \in \N$ we let $p_n : [-1,1] \to \R$ be the $n^{th}$ Legendre polynomial, normalized so that $\ip{p_n,p_m}_{L^2} = \delta_{n,m}$.  Given $u \in L^2((-1,1))$ we define $\hat{u} : \N \to \R$ via $\hat{u}(n) = \ip{u,p_n}_{L^2}$.  We write $\mathcal{P} = \{ p : (-1,1) \to \R \st p \text{ is a polynomial}\}$.  It is well-known (see, for instance, \cite{Sz_1978})  that the Legendre polynomials span $\mathcal{P}$ and form a complete orthonormal set in $L^2((-1,1))$. 
\end{enumerate}
\end{dfn}

The following result records some essential properties of the $\h^k$ spaces proved in \cite{ELW_2002, LW_2002, LW_2013}; we also refer to \cite{FGHL_2025}.

\begin{thm}\label{hk_basics}
The following hold for $k \in \N$.
\begin{enumerate}
 \item The space $\h^k$ is complete, and hence a Hilbert space.  Additionally, if $k \le m \in \N$ then we have the continuous inclusion $\h^m \hookrightarrow \h^k$.
 \item The set $\mathcal{P} \subset \h^k$ is dense.   
 \item Let $\lambda \in \R_+$.  Then 
\begin{equation}\label{hk_basics_00}
 \h^k = \{ u \in L^2((-1,1)) \st \sum_{n=0}^\infty (n(n+1) + \lambda)^k \abs{\hat{u}(n)}^2  < \infty\},
\end{equation}
and we have the norm equivalence
\begin{equation}\label{hk_basics_01}
 \norm{u}_{\h^k} \asymp \left(\sum_{n=0}^\infty (n(n+1) + \lambda )^k \abs{\hat{u}(n)}^2  \right)^{1/2} .
\end{equation}
Moreover, $\{p_n / [n(n+1) +\lambda]^{k/2} \}_{n \in \N}$ is a complete orthonormal set in $\h^k$ when it is endowed with the inner-product associated to the square-norm on the right of \eqref{hk_basics_01}.

 \item The Legendre differential operator $L: \mathcal{P} \to \mathcal{P}$ defined by 
\begin{equation}
 Lu(\zeta) = -(1-\zeta^2) D^2 u(\zeta) + 2\zeta Du(\zeta)  = -D((1-\zeta^2) Du(\zeta))
\end{equation}
uniquely extends to a bounded linear map $L : \h^{k+2} \to \h^k$.  Moreover, given any $r \in \R \backslash \{ n (n+1) \st n \in \N\}$ the map  $L - r I : \h^{k+2} \to \h^k$  is an isomorphism.  
 \item The map $\mathcal{I} : \h^{k} \to \h^{k+1}$ given by $\mathcal{I} u(\zeta) = \int_0^\zeta u(z) dz$ is well-defined, bounded, and linear.
\end{enumerate}

\end{thm}
\begin{proof}
The first three items are proved in \cite{ELW_2002,LW_2002,LW_2013}.  The fourth follows easily from the third and the fact that the Legendre polynomials satisfy $L p_n = n(n+1) p_n$.  The fifth follows since $D\mathcal{I}u = u$ and $\norm{\mathcal{I} u}_{\h^{k+1}}^2 \le \norm{\mathcal{I} u}_{L^2}^2 + \norm{D\mathcal{I} u}_{\h^{k}}^2 \le \norm{u}_{L^2}^2 + \norm{u}_{\h^k}^2 \le 2 \norm{u}_{\h^k}^2$, which follows from the definition of the norm, the bound $(1-\zeta^2)^{j+1} \le (1-\zeta^2)^{j}$ for $\zeta \in [-1,1]$ and $j \in \N$, and Cauchy-Schwarz.
\end{proof}

Next, we aim to establish some useful functional analytic properties of the $\h^k$ spaces.  We begin with a trio of embedding results.  The first relates to the uniformly weighted spaces discussed previously.

\begin{thm}\label{hk_weight_improvement}
Let $k \in \N$.  Then the following hold.
\begin{enumerate}
 \item We have the continuous inclusions 
\begin{equation}
\h^k \hookrightarrow H^k_k((-1,1)) \hookrightarrow  H^{\lfloor k/2\rfloor}((-1,1)).
\end{equation}
In particular, if $k \ge 2$ then $\h^k \hookrightarrow C^{\lfloor k/2\rfloor -1, 1/2}_b([-1,1])$.
 
 \item If $1 \le k$ and $u \in \h^k$, then for $0 \le j \le k-1$ we have that $D^j u \in H^0_{\max\{-1,2j-k\}}((-1,1))$ and 
\begin{equation}
 \sum_{j=0}^{k-1} \norm{D^j u}_{H^0_{\max\{-1,2j-k\}}} \ls \norm{u}_{\h^k}.
\end{equation}
\end{enumerate}
\end{thm}
\begin{proof}
The first item follows from  Theorem \ref{weighted_space_properties}, the fact that $\mathcal{P} \subseteq C^\infty([-1,1])$, and the standard Sobolev embedding in one dimension.   The second item also  follows from  Theorem  \ref{weighted_space_properties} and a simple induction argument.
\end{proof}

The second gives some $H^1$ bounds (and hence $C^{0,1/2}_b$ bounds) for various derivatives weighted by powers of $(1-\zeta^2)$.

\begin{prop}\label{H1C0_bound}
Suppose that $u \in \h^k$ for $1 \le k \in \N$, and let $0 \le j \le k-1$.
\begin{enumerate}
 \item If $j \le (k-1)/2$ then for $\mu >1/2$ we have that  $(1-\zeta^2)^\mu D^j u \in H^1((-1,1))$ and $\norm{(1-\zeta^2)^\mu D^j u }_{H^1} \ls \norm{u}_{\h^k}$.
 \item If $(k-1)/2 < j$ then we have that  $(1-\zeta^2)^{j+1-k/2} D^j u \in H^1((-1,1))$ and $\norm{(1-\zeta^2)^{j+1-k/2} D^j u }_{H^1} \ls \norm{u}_{\h^k}$. 
\end{enumerate}
Moreover, since $H^1((-1,1)) \hookrightarrow C^{0,1/2}_b([-1,1])$, each of the above $H^1$ estimates also yields a $C^{0,1/2}_b$ estimate.
\end{prop}
\begin{proof}
Fix $j,k \in \N$ with $0 \le j \le k-1$.  Write $\mu_j = j+1-k/2$ when $(k-1)/2 < j$, and let $\mu_j >1/2$ be arbitrary when $j \le (k-1)/2$.  Then, regardless of the value of $j$, we have the inequalities
\begin{equation}\label{H1C0_bound_0}
 \hal < \mu_j \text{ and } \frac{2j-k+2}{2} \le \mu_j.
\end{equation}

Next, we compute 
\begin{equation}
 D[(1-\zeta^2)^{\mu_j} D^j u] = (1-\zeta^2)^{\mu_j} D^{j+1}u  -2 \mu_j \zeta (1-\zeta^2)^{\mu_j-1}    D^j u
\end{equation}
in order to bound 
\begin{equation}
 \norm{(1-\zeta^2)^{\mu_j} D^j u}_{H^1}^2 \ls \int_{-1}^1 (1-\zeta^2)^{2\mu_j-2} \abs{D^j u}^2 +  (1-\zeta^2)^{2\mu_j} \abs{D^{j+1} u}^2.
\end{equation}
In light of \eqref{H1C0_bound_0} and Theorem \ref{hk_weight_improvement}, we may then bound 
\begin{equation}
\int_{-1}^1 (1-\zeta^2)^{2\mu_j-2} \abs{D^j u}^2 +  (1-\zeta^2)^{2\mu_j} \abs{D^{j+1} u}^2 \ls  \norm{D^j u}_{H^0_{\max\{-1,2j-k\}}}^2 +  \norm{D^{j+1} u}_{H^0_{\max\{-1,2j+2-k\}}}^2 \ls \norm{u}_{\h^k}^2.
\end{equation}
Combining the previous two bounds then yields the desired inequality.
\end{proof}

Our third embedding, a variant of Proposition \ref{H1C0_bound}, will be useful as well.

\begin{prop}\label{hkhm_bound_spec}
Suppose that $u \in \h^k$ for $1 \le k \in \N$, and suppose $j,m,\mu \in \N$  satisfy  $j+m \le k$ and $2j +m \le k+2\mu$.  Then    $(1-\zeta^2)^\mu D^j u \in \h^m$ and $\norm{(1-\zeta^2)^\mu D^j u }_{\h^m} \ls \norm{u}_{\h^k}$.
\end{prop}
\begin{proof}
Fix $k,j,m,\mu$ as in the hypotheses.  We then use the Leibniz rule  to bound
\begin{equation}\label{hkhm_bound_spec_1}
 \int_{-1}^1 (1-\zeta^2)^\ell \abs{ D^\ell [ (1-\zeta^2)^\mu D^j u]   }^2 \ls  \sum_{n=0}^\ell \int_{-1}^1 (1-\zeta^2)^{\ell + 2 \max\{\mu-n,0\} } \abs{D^{j+\ell-n}u}^2 \text{ for } 0 \le \ell \le m.
\end{equation}
Breaking to cases based on whether $0\le \mu \le n-1$ or $n \le \mu$, we readily check that the assumptions on $j,m,\mu$ imply that 
\begin{equation}
 \max\{-1,2(j+\ell-n)-k\} \le \ell + 2 \max\{\mu-n,0\} \text{ for } 0 \le \ell \le m \text{ and } 0 \le n \le \ell,
\end{equation}
and hence Theorem \ref{hk_weight_improvement} allows us to bound 
\begin{equation}\label{hkhm_bound_spec_2}
 \int_{-1}^1 (1-\zeta^2)^{\ell + 2 \max\{\mu-n,0\} } \abs{D^{j+\ell-n}u}^2  \ls \norm{D^{j+\ell-n} u}_{H^0_{\max\{-1,2(j+\ell-n)-k\}}}^2 \ls \norm{u}_{\h^k}^2
\end{equation}
for $0 \le \ell \le m$ and $0 \le n \le \ell$.  Combining \eqref{hkhm_bound_spec_1} and \eqref{hkhm_bound_spec_2} then yields the desired inequality.
\end{proof}

Next we prove that when $2 \le k \in \N$ the space $\h^k$ is an algebra.

\begin{thm}\label{hs_algebra}
The space $\h^k$ is an algebra when $2 \le k \in \N$, i.e. if $u,v \in \h^k$ then $uv \in \h^k$ and $\norm{uv}_{\h^k} \ls \norm{u}_{\h^k} \norm{v}_{\h^k}$. 
\end{thm}
\begin{proof}
We proceed by induction on $k \ge 2$.  Suppose that $u,v \in \h^2$.  Using the $L^\infty$ bounds from Theorem \ref{hk_weight_improvement}, we may estimate 
\begin{equation}\label{hs_algebra_1}
 \int_{-1}^1 \abs{uv}^2 \ls \norm{u}_{L^\infty}^2 \norm{v}_{L^2}^2 \ls \norm{u}_{\h^2}^2 \norm{v}_{\h^2}^2
\end{equation}
and 
\begin{equation}\label{hs_algebra_2}
 \int_{-1}^1 (1-\zeta^2) [\abs{u}^2 \abs{Dv}^2 + \abs{Du}^2 \abs{v}^2  ]  \ls 
 \norm{u}_{L^\infty}^2 \norm{v}_{\h^1}^2 + \norm{u}_{\h^1}^2 \norm{v}_{L^\infty}^2 
 \ls  \norm{u}_{\h^2}^2 \norm{v}_{\h^2}^2.
\end{equation}
On the other hand, we can apply Proposition \ref{H1C0_bound} (with $k=2$ and $j=1$) and Theorem \ref{hk_weight_improvement}  to bound 
\begin{equation}
 \norm{ (1-\zeta^2) D u }_{L^\infty} \ls \norm{u}_{\h^2} \text{ and } \norm{Dv}_{L^2} \ls \norm{v}_{\h^2},
\end{equation}
and so
\begin{multline}\label{hs_algebra_3}
 \int_{-1}^1 (1-\zeta^2)^2 [\abs{u}^2 \abs{D^2 v}^2 +  \abs{Du}^2 \abs{Dv}^2 + \abs{D^2 u}^2 \abs{v}^2  ]  \ls 
\norm{u}_{L^\infty}^2 \norm{v}_{\h^2}^2 +  \norm{u}_{\h^2}^2 \norm{v}_{L^\infty}^2 \\
+  \int_{-1}^1 (1-\zeta^2)^2  \abs{Du}^2 \abs{Dv}^2  
\ls \norm{u}_{\h^2}^2 \norm{v}_{\h^2}^2 +  \norm{ (1-\zeta^2) D u }_{L^\infty}^2 \norm{Dv}_{L^2}^2 \ls \norm{u}_{\h^2}^2 \norm{v}_{\h^2}^2.
\end{multline}
Combining \eqref{hs_algebra_1}, \eqref{hs_algebra_2}, and \eqref{hs_algebra_3} with the Leibniz rule then shows that $uv \in \h^2$ with $\norm{uv}_{\h^2} \ls \norm{u}_{\h^2} \norm{v}_{\h^2}$.  This completes the proof of the base case of the induction.

Suppose now that the result is proved for $2 \le k$ and suppose that $u,v \in \h^{k+1}$.  Then $u,v \in \h^k$, so the induction hypothesis implies that $uv \in \h^k$ with $\norm{uv}_{\h^k} \ls \norm{u}_{\h^k} \norm{v}_{\h^k}$.  Note that for $w \in \acl^{k}((-1,1))$ we have that 
\begin{equation}\label{hs_algebra_4}
 \norm{w}_{\h^{k+1}}^2 = \norm{w}_{\h^k}^2 + \int_{-1}^1 (1-\zeta^2)^{k+1} \abs{D^{k+1} w}^2 
\end{equation}
so we only need to handle the last term with $w = uv$.  We use the Leibniz rule to bound 
\begin{equation}\label{hs_algebra_5}
 \int_{-1}^1 (1-\zeta^2)^{k+1} \abs{D^{k+1}(uv)}^2 \ls   \sum_{j=0}^{k+1} \int_{-1}^1 (1-\zeta^2)^{k+1} \abs{D^j u}^2  \abs{D^{k+1-j} v}^2. 
\end{equation}
We now break to cases based on the value of $j$.

We begin with the cases $j \in \{0,k+1\}$, in which case we use Theorem \ref{hk_weight_improvement} to estimate
\begin{multline}
 \int_{-1}^1 (1-\zeta^2)^{k+1} \abs{u}^2  \abs{D^{k+1} v}^2 +  \int_{-1}^1 (1-\zeta^2)^{k+1} \abs{D^{k+1} u}^2  \abs{ v}^2  \ls \norm{u}_{L^\infty}^2 \norm{v}_{\h^{k+1}}^2 +  \norm{u}_{\h^{k+1}}^2 \norm{v}_{L^\infty}^2 \\
 \ls \norm{u}_{\h^{k+1}}^2 \norm{v}_{\h^{k+1}}^2.
\end{multline}

In the second case we consider
\begin{equation}
 1 \le j \le \frac{(k+1)-1}{2} = \frac{k}{2}.
\end{equation}
In this setting we easily check that $1 \le k+1-2j$, and hence
\begin{equation}
 k+1-2j = \max\{-1, 2(k+1-j) -k -1\}.
\end{equation}
Using Proposition \ref{H1C0_bound} with $\mu =j$ and Theorem \ref{hk_weight_improvement}, we have 
\begin{equation}
 \norm{(1-\zeta^2)^j D^j u}_{L^\infty} \ls \norm{u}_{\h^{k+1}}  \text{ and } \norm{D^{k+1-j} v}_{H^0_{ k+1-2j }} \ls \norm{v}_{\h^{k+1}}.
\end{equation}
Thus, 
\begin{equation}
\int_{-1}^1 (1-\zeta^2)^{k+1} \abs{D^j u}^2  \abs{D^{k+1-j} v}^2 \ls \norm{u}_{\h^{k+1}}^2  \int_{-1}^1 (1-\zeta^2)^{k+1-2j}   \abs{D^{k+1-j} v}^2 \ls \norm{u}_{\h^{k+1}}^2 \norm{v}_{\h^{k+1}}^2. 
\end{equation}

Next, consider the case 
\begin{equation}
 \frac{k}{2}+1 \le j \le k.
\end{equation}
This requires that 
\begin{equation}
1 \le k+1-j \le \frac{k}{2},
\end{equation}
and so we may use the same argument from the previous case with $u$ and $v$ swapped to see that 
\begin{equation}
 \int_{-1}^1 (1-\zeta^2)^{k+1} \abs{D^j u}^2  \abs{D^{k+1-j} v}^2 \ls \norm{u}_{\h^{k+1}}^2 \norm{v}_{\h^{k+1}}^2.
\end{equation}

The only remaining case is to handle 
\begin{equation}
 \frac{k}{2} < j < \frac{k}{2} + 1,
\end{equation}
which occurs only when $k$ is odd.  In this setting Proposition \ref{H1C0_bound} now shows that 
\begin{equation}
  \norm{(1-\zeta^2)^{j+1/2 -k/2} D^j u}_{L^\infty} \ls \norm{u}_{\h^{k+1}}.
\end{equation}
From this and the fact that $2 \le k$ and $j < k/2 +1$ imply that $k+1-j \le 2k-2j$, we then see that
\begin{multline}
  \int_{-1}^1 (1-\zeta^2)^{k+1} \abs{D^j u}^2  \abs{D^{k+1-j} v}^2 \ls \norm{u}_{\h^{k+1}}^2  \int_{-1}^1 (1-\zeta^2)^{2k-2j}   \abs{D^{k+1-j} v}^2  \\
\ls \norm{u}_{\h^{k+1}}^2  \int_{-1}^1 (1-\zeta^2)^{k+1-j}   \abs{D^{k+1-j} v}^2 \ls \norm{u}_{\h^{k+1}}^2  \norm{v}_{\h^{k+1}}^2. 
\end{multline}

Synthesizing the analysis from these four cases then shows that
\begin{equation}
  \sum_{j=0}^{k+1} \int_{-1}^1 (1-\zeta^2)^{k+1} \abs{D^j u}^2  \abs{D^{k+1-j} v}^2 \ls \norm{u}_{\h^{k+1}}^2  \norm{v}_{\h^{k+1}}^2.
\end{equation}
Combining this with \eqref{hs_algebra_4}, \eqref{hs_algebra_5}, and the induction hypothesis, we then deduce that $uv \in \h^{k+1}$ and $\norm{uv}_{\h^{k+1}} \ls \norm{u}_{\h^{k+1}}  \norm{v}_{\h^{k+1}}$.  This completes the inductive step, and hence the proof.
\end{proof}

A particular consequence of Theorem \ref{hs_algebra} is that if $\alpha_0,\dotsc,\alpha_m \in \R$ then for $2\le k \in \N$  the polynomial map  $\h^k \ni u \mapsto \sum_{\ell=0}^m \alpha_\ell u^\ell \in \h^k$ is smooth.  We next aim to establish analogous results for more general composition maps.  We begin with a crucial technical result showing composition is well-defined and enjoys useful estimates.

\begin{thm}\label{hk_composition}
Let $\varnothing \neq E \subseteq U \subseteq \R$ with $U$ open, and let $\varnothing \neq \mathcal{E}_E \subseteq \h^2$ be such that $u \in \mathcal{E}_E$ implies that $u((-1,1)) \subseteq E$.  Suppose that $2 \le k\in \N$ and $f \in C^k(U)$ satisfies 
\begin{equation}
 \norm{f}_{C^k_b(E)} = \sup_{z \in E} \sup_{0 \le j \le k} \abs{D^j f(z)} < \infty.
\end{equation}
Then for  $u \in \h^k \cap \mathcal{E}_E$ we have that $f \circ u \in \h^k$ and 
\begin{equation}
 \norm{f\circ u}_{\h^k} \ls \norm{f}_{C^k_b(E)}  (1+\norm{u}_{\h^k})^k.
\end{equation}
Here the implicit constant is independent of $f$ and $u$.
\end{thm}
\begin{proof}
We will prove the result by inducting on $2\le k \in \N$.  Throughout the proof any implicit constants in inequalities of the form $a \ls b$ are independent of $f$ and $u$ but may depend on $k$ and various other universal constants.  We will also abbreviate $\norm{f}_{C^k_b} = \norm{f}_{C^k_b(E)}$, noting that 
\begin{equation}
\sup_{\zeta \in (-1,1)}  \abs{D^jf(u(\zeta))} \le \sup_{z\in E} \abs{D^j f(z)}  \le \norm{f}_{C^k_b} 
\end{equation}
for $0 \le j \le k$ and $u\in \mathcal{E}_E$.

We begin with the case $k=2$, assuming that $u \in \mathcal{E}_E$.  We first bound 
\begin{equation}
 \int_{-1}^1 \abs{f\circ u}^2 \ls \norm{f}_{C^2_b}^2
\end{equation}
and 
\begin{equation}
 \int_{-1}^1 (1-\zeta^2) \abs{Df(u) Du}^2 \le \norm{f}_{C^2_b}^2 \int_{-1}^1 (1-\zeta^2) \abs{Du}^2 \le  \norm{f}_{C^2_b}^2 \norm{u}_{\h^2}^2.
\end{equation}
We then use Proposition \ref{H1C0_bound} (with $k=2$ and $j=1$) and Theorem \ref{hk_weight_improvement} to bound 
\begin{multline}
  \int_{-1}^1 (1-\zeta^2)^2 [\abs{Df(u) D^2 u}^2 + \abs{D^2f(u) (Du)^2 }^2  ] \ls \norm{f}_{C^2_b}^2 \int_{-1}^1 (1-\zeta^2)^2 \abs{D^2u}^2 \\
  +  \norm{f}_{C^2_b}^2 \norm{(1-\zeta^2) D u}_{L^\infty}^2 \int_{-1}^1 \abs{Du}^2 \ls \norm{f}_{C^2_b}^2(\norm{u}_{\h^2}^2 + \norm{u}_{\h^2}^4   ).
\end{multline}
Upon combining these three bounds and using the chain rule, we then find that 
\begin{equation}
 \norm{f\circ u}_{\h^2}^2 \ls \norm{f}_{C^2_b}^2  (1+\norm{u}_{\h^2})^4.
\end{equation}
This completes the proof of the base case $k=2$.

Now suppose that $k \ge 3$ and that the result has been proved for $k-1 \ge 2$.  Consider $u \in \h^k \cap \mathcal{E}_E$. Since $u \in \h^{k-1}$ we then have that  $f \circ u \in \h^{k-1}$ and  
\begin{equation}
 \norm{f\circ u}_{\h^{k-1}}^2 \ls \norm{f}_{C^{k-1}_b}^2  (1+\norm{u}_{\h^{k-1}})^{2(k-1)}.
\end{equation}
Thus, to establish the desired result for $k$ it suffices to prove that 
\begin{equation}\label{hk_composition_1}
 \int_{-1}^1 (1-\zeta^2)^k \abs{ D^k  [ f\circ u ] }^2  \ls \norm{f}_{C^k_b}^2  (1+\norm{u}_{\h^k})^{2k}.
\end{equation}

To prove \eqref{hk_composition_1} we begin by employing the Fa\`a di Bruno formula to bound
\begin{equation}\label{hk_composition_2}
 \int_{-1}^1 (1-\zeta^2)^k \abs{ D^k  [ f\circ u ] }^2 \ls \sum_{m=1}^k \sum_{j \in J(m,k)} \int_{-1}^1 (1-\zeta^2)^k \abs{D^mf(u)}^2 \abs{D^{j_1} u \cdots D^{j_m} u }^2,
\end{equation}
where 
\begin{equation}
 J(m,k) = \{j \in \{1,\dotsc,k\}^m \st j_1 +\cdots + j_m = k \text{ and } j_1 \le j_2 \le \dotsc \le j_m\}.
\end{equation}
Note that for $j \in J(m,k)$ we always have the bounds 
\begin{equation}\label{hk_composition_3}
 j_1 \le \frac{k}{m} \le j_m \le k -  m+1.
\end{equation}
Next, we will break to three cases to estimate the right side of \eqref{hk_composition_2}.  

\emph{Case 1,  $m=1$:}   Assume that $m=1$.  Then we may bound the right side of \eqref{hk_composition_2} via 
\begin{multline}\label{hk_composition_3_5}
 \sum_{j \in J(1,k)} \int_{-1}^1 (1-\zeta^2)^k \abs{D^mf(u)}^2 \abs{D^{j_1} u \cdots D^{j_m} u }^2 = \int_{-1}^1 (1-\zeta^2)^k \abs{Df(u)}^2 \abs{D^{k} u}^2 \\
 \le \norm{f}_{C^k_b}^2 \int_{-1}^1 (1-\zeta^2)^k \abs{D^{k} u}^2 \le \norm{f}_{C^k_b}^2 \norm{u}_{\h^k}^2.
\end{multline}

\emph{Case 2,  $m=2$:} In the second case we assume that $m=2$.  Then $j \in J(2,k)$ means that $j=(j_1,j_2)$ with $j_1 \le j_2$ and $j_1 + j_2 =k$.  This implies that 
\begin{equation}
 j_1 \le \frac{k}{2} \le j_2 \text{ and hence that } \frac{k-1}{2} < \frac{k}{2} \le j_2.
\end{equation}
We break to two subcases.  In the first we assume that $j_1 \le (k-1)/2$ and set $\mu = 1 > 1/2$, which then implies that    
\begin{equation}
 -1 < k -2\mu \text{ and } 2j_2 -k \le  k-2\mu.
\end{equation}
We may then use the first item of Proposition \ref{H1C0_bound} on the $j_1$ term with $\mu=1$, and then use  Theorem \ref{hk_weight_improvement} on the $j_2$ term in order to bound  
\begin{multline}\label{hk_composition_4}
 \int_{-1}^1 (1-\zeta^2)^k \abs{D^2 f(u)}^2 \abs{D^{j_1} u  D^{j_2} u }^2 \ls \norm{f}_{C^k_b}^2 \norm{(1-\zeta^2) D^{j_1} u}_{L^\infty}^2   \int_{-1}^1 (1-\zeta^2)^{k-2\mu}   \abs{D^{j_2} u }^2  \\
 \ls  \norm{f}_{C^k_b}^2 \norm{u}_{\h^k}^2 \norm{D^{j_2} u}_{H^0_{\max\{-1,2 j_2 -k\}}}^2  \ls  \norm{f}_{C^k_b}^2 \norm{u}_{\h^k}^4.
\end{multline}
This completes the first subcase.

In the second subcase we assume that $(k-1)/2 < j_1$.   Since $2 \le k$ we have that 
\begin{equation}
 j_1 \le \frac{k}{2} < k - \hal \text{ and }  2(k-j_1) +2j_1 \le 3k-2,
\end{equation}
which imply that
\begin{equation}
 -1 < 2(k-j_1-1) \text{ and } 2j_2 -k \le  2(k-j_1-1).
\end{equation}
Consequently, we may use the second item of Proposition \ref{H1C0_bound} on the $j_1$ term and 
Theorem \ref{hk_weight_improvement} on the $j_2$ term in order to bound  
\begin{multline}\label{hk_composition_5}
 \int_{-1}^1 (1-\zeta^2)^k \abs{D^2 f(u)}^2 \abs{D^{j_1} u  D^{j_2} u }^2 \ls \norm{f}_{C^k_b}^2 \norm{(1-\zeta^2)^{j_1+1-k/2} D^{j_1} u}_{L^\infty}^2   \int_{-1}^1 (1-\zeta^2)^{k-2(j_1+1-k/2)} \abs{D^{j_2} u }^2  \\
 \ls  \norm{f}_{C^k_b}^2 \norm{u}_{\h^k}^2 \norm{D^{j_2} u}_{H^0_{\max\{-1,2 j_2 -k\}}}^2  \ls  \norm{f}_{C^k_b}^2 \norm{u}_{\h^k}^4.
\end{multline} 
This completes the second subcase.

The two subcases cover all possibilities when $m=2$.  We then deduce from \eqref{hk_composition_4} and \eqref{hk_composition_5} that 
\begin{equation}\label{hk_composition_6}
\sum_{j \in J(2,k)} \int_{-1}^1 (1-\zeta^2)^k \abs{D^2f(u)}^2 \abs{D^{j_1} u  D^{j_2} u }^2 \ls \norm{f}_{C^k_b}^2 \norm{u}_{\h^k}^4.
\end{equation}

\emph{Case 3,  $3 \le m \le k$:}  In the third case we assume that $3 \le m \le k$.  If $j \in J(m,k)$ satisfies $(k-1)/2 < j_{m-1}$, then 
\begin{equation}
 k-1 < j_{m-1} + j_m \le k-1,
\end{equation}
a contradiction.  Thus, for all $j \in J(m,k)$ we must have that $j_{m-1} \le (k-1)/2$.  Next, we break to two subcases based on the size of $j_m$.

In the first subcase we assume that $j_m \le (k-1)/2$.  We may then use the first item of Proposition \ref{H1C0_bound} on each of the $j_i$ terms with 
\begin{equation}
 \mu = \frac{k+1/2}{2m} > \frac{k}{2m} \ge \hal,
\end{equation}
in order to bound 
\begin{multline}\label{hk_composition_7}
 \int_{-1}^1 (1-\zeta^2)^k \abs{D^mf(u)}^2 \abs{D^{j_1} u \cdots D^{j_m} u }^2 \ls \norm{f}_{C^k_b}^2 \prod_{i=1}^m  \norm{(1-\zeta^2)^{\mu} D^{j_i} u}_{L^\infty}^2  
 \int_{-1}^1 (1-\zeta^2)^{k-2m\mu}  \\
 \ls  \norm{f}_{C^k_b}^2 \norm{u}_{\h^k}^{2m} \int_{-1}^1 (1-\zeta^2)^{-1/2} \ls   \norm{f}_{C^k_b}^2 \norm{u}_{\h^k}^{2m}.
\end{multline}
This completes the first subcase.

In the second subcase we assume that $(k-1)/2 < j_m$.  Since $m < k+2$ we have that 
\begin{equation}
\hal < \frac{k+1}{2(m-1)},
\end{equation}
and this then allows us to pick $\mu \in (1/2,1]$ satisfying 
\begin{equation}
 \mu < \frac{k+1}{2(m-1)}.
\end{equation}
This and \eqref{hk_composition_3} require 
\begin{equation}
 -1 < k -2(m-1) \mu \text{ and } 2j_{m} -k \le k - 2(m-1)\mu, 
\end{equation}
so we may again apply the first item of Proposition \ref{H1C0_bound} with this choice of $\mu$ to the $j_1,\dotsc,j_{m-1}$ terms and Theorem \ref{hk_weight_improvement} on the $j_m$ term to bound 
\begin{multline}\label{hk_composition_8}
 \int_{-1}^1 (1-\zeta^2)^k \abs{D^mf(u)}^2 \abs{D^{j_1} u \cdots D^{j_m} u }^2 \ls  \norm{f}_{C^k_b}^2 \prod_{i=1}^{m-1}  \norm{(1-\zeta^2)^{\mu} D^{j_i} u}_{L^\infty}^2  \int_{-1}^1 (1-\zeta^2)^{k-2(m-1)\mu}   \abs{D^{j_{m}} u }^2  \\
  \ls 
\norm{f}_{C^k_b}^2 \norm{u}_{\h^k}^{2m-2}   
  \norm{D^{j_{m}} u}_{H^0_{\max\{-1,2 j_{m} -k\}}}^2  \ls \norm{f}_{C^k_b}^2 \norm{u}_{\h^k}^{2m}.   
\end{multline}
This completes the second subcase.

The two subcases cover all possibilities when $3 \le m\le k$.  We then deduce from \eqref{hk_composition_7} and \eqref{hk_composition_8} that 
\begin{equation}\label{hk_composition_9}
\sum_{m=3}^k \sum_{j \in J(m,k)} \int_{-1}^1 (1-\zeta^2)^k \abs{D^mf(u)}^2 \abs{D^{j_1} u \cdots D^{j_m} u }^2 \ls \norm{f}_{C^k_b}^2 \sum_{m=3}^k \norm{u}_{\h^k}^{2m}.
\end{equation}

The three cases considered above cover all possibilities of $1 \le m \le k$.  Hence, we may combine \eqref{hk_composition_2} with \eqref{hk_composition_3_5}, \eqref{hk_composition_6}, and \eqref{hk_composition_9} to conclude that 
\begin{equation}
  \int_{-1}^1 (1-\zeta^2)^k \abs{ D^k  [ f\circ u ] }^2 \ls \norm{f}_{C^k_b}^2 (1+\norm{u}_{\h^k})^{2k}.
\end{equation}
This is \eqref{hk_composition_1}, which completes the proof of the induction step, and hence the proof.
\end{proof}

We now prove that composition with certain functions induces a $C^m$ map on $\h^k$ for $k \ge 2$.  This is in some sense analogous to the well-known ``$\omega$-lemmas'': see Proposition 2.4.18 in \cite{AMR_1988} for the result in $C^0$ spaces, and \cite{IKT_2013} for results in standard Sobolev spaces.  This result is essential for our use of the $\h^k$ spaces in the Crandall-Rabinowitz framework.

\begin{thm}\label{hk_comp_Cm}
Let $k,m \in \N$ with $k \ge 2$, and let $\varnothing \neq U \subseteq \R$ be open and $f \in C^{k+m+1}(U)$.  Define the open set 
\begin{equation}
 \h^2[U] = \{u \in \h^2 \st \overline{u((-1,1))} \subset U \}.
\end{equation}
Then the map $\Phi : \h^k \cap \h^2[U] \to \h^k$ defined by $\Phi(u) = f\circ u$ is well-defined and $C^m$.  Moreover, if $m \ge 1$ then for $u \in \h^k \cap \h^2[U]$ and $1 \le j \le m$ we have that $D^j \Phi(u) \in \L^j(\h^k;\h^k)$ is given by  
\begin{equation}\label{hk_comp_Cm_0}
 D^j\Phi(u) (w_1, \dotsc, w_j)  = D^j f(u) \prod_{i=1}^j w_i.
\end{equation}
\end{thm}
\begin{proof}
We divide the proof into steps.

\emph{Step 1 -- Well-definedness and a reduction:}   Write $\mathcal{K}(U) = \{ K \subset U \st K \text{ is compact and } K^\circ \neq \varnothing\}$.  Given $K \in \mathcal{K}(U)$ let 
\begin{equation}
 \mathcal{E}_K = \{u \in \h^2 \st \overline{u((-1,1))} \subseteq K^\circ \},
\end{equation}
and note that since $\h^2 \hookrightarrow C^{0,1/2}_b((-1,1))$ we have that each   $\mathcal{E}_K \subseteq \h^2$ is open  and 
\begin{equation}
 \h^2[U] = \bigcup_{K \in \mathcal{K}(U)} \mathcal{E}_K.
\end{equation}
Moreover, for $K \in \mathcal{K}(U)$ we have that $u \in \mathcal{E}_K$ implies that $\overline{u((-1,1))} \subseteq K$, so Theorem \ref{hk_composition} implies that if $F\in C^k(U)$ then  $F\circ u \in \h^k$ for all $u \in \h^k \cap \mathcal{E}_K$.  From this we learn two key facts.  First, the map $\Phi : \h^k \cap \h^2[U] \to \h^k $ is indeed well-defined.  Second, it suffices to prove that the restriction $\Phi : \h^k \cap \mathcal{E}_K \to \h^k$ is $C^m$ and satisfies \eqref{hk_comp_Cm_0} (when $m \ge 1$) for each $K \in \mathcal{K}(U)$.

\emph{Step 2 -- Continuity:}  Fix $K \in \mathcal{K}(U)$.   We now claim that $\Phi : \h^k \cap \mathcal{E}_K \to \h^k$ is continuous.  Fix $u \in B(u,r) \subset \h^k \cap \mathcal{E}_K$.   For $v \in B(0,r) \backslash \{0\}$ we then compute 
\begin{equation}
\Phi(u+v) - \Phi(u) = f(u+v) - f(u) = \int_0^1 Df(u+tv)v dt. 
\end{equation}
Using Minkowski's integral inequality,  Theorem \ref{hk_composition},  and the fact that $\h^k$ is an algebra, we bound
\begin{multline}
 \norm{\Phi(u+v) - \Phi(u)}_{\h^k} \le \int_0^1 \norm{Df(u+tv) v}_{\h^k} dt \ls \norm{v}_{\h^k} \norm{f}_{C^{k+1}_b(K)} \int_0^1 (1+ \norm{u+tv}_{\h^k})^k dt \\
 \ls  \norm{v}_{\h^k} \norm{f}_{C^{k+1}_b(K)} (1+ \norm{u}_{\h^k} + r)^k,
\end{multline}
from which we readily deduce that $\Phi$ is continuous at $u$.   This completes the proof of the claim and the theorem as well when $m=0$.  It remains only to handle the case when $m \ge 1$.

\emph{Step 3 -- Differentiability:}  Suppose that $m\ge 1$ and fix $K \in \mathcal{K}(U)$.   We saw in Step 1 that $f(u) \in \h^k$ when $u \in \h^k \cap \mathcal{E}_K$, but since $f \in C^{k+m+1}(U)$ the same reasoning shows that $D^j f(u) \in \h^k$ when $1 \le j \le m+1$ and $u \in \h^k \cap \mathcal{E}_K$.  Moreover, since $\h^k$ is an algebra, the map $\h^k \ni w \mapsto D^j f(u) w \in \h^k$ is also bounded and linear, and Theorem \ref{hk_composition} provides the estimate
\begin{equation}\label{hk_comp_Cm_1}
\max_{0\le j \le m+1} \norm{D^j f(u) w}_{\h^k} \ls \norm{f}_{C^{k+m+1}_b(K)}(1+\norm{u}_{\h^k})^k \norm{w}_{\h^k} \text{ for } u\in \h^k \cap \mathcal{E}_K \text{ and } w \in \h^k.
\end{equation}
In particular, this shows that the proposed formulas for the derivatives of $\Phi$ in \eqref{hk_comp_Cm_0} are all well-defined.

We now aim to inductively prove that $\Phi : \h^k \cap \mathcal{E}_K \to \h^k$ is $m$-times differentiable. To this end fix $u \in B(u,r) \subseteq \mathcal{E}_K \cap \h^k$.   Note that \eqref{hk_comp_Cm_1} implies that the map $\h^k \ni v \mapsto Df(u) v \in \h^k$ is bounded and linear. For $v \in B(0,r) \backslash \{0\}$ we then compute 
\begin{multline}
\Phi(u+v) - \Phi(u) - Df(u) v = f(u+v) - f(u) - Df(u) v = \int_0^1 [Df(u+tv) - Df(u)]v dt \\
= \int_0^1 \int_0^1 t D^2f(u+stv) v^2 ds dt. 
\end{multline}
Again using Theorem \ref{hk_composition} and the fact that $\h^k$ is an algebra, we then have that 
\begin{multline}
\norm{\Phi(u+v) - \Phi(u) - Df(u) v}_{\h^k} \le \int_0^1 \int_0^1 t \norm{D^2f(u+stv) v^2}_{\h^k} ds dt \\
\ls \norm{v}_{\h^k}^2 \int_0^1 \int_0^1 t \norm{f}_{C^{k+2}_b(K)}(1+ \norm{u+st v}_{\h^k} )^k dsdt 
\ls  \norm{v}_{\h^k}^2  \norm{f}_{C^{k+2}_b(K)} (1+ \norm{u}_{\h^k} + r)^k.
\end{multline}
Hence, 
\begin{equation}
 \lim_{v \to 0} \frac{\norm{\Phi(u+v) - \Phi(u) - Df(u) v}_{\h^k}}{\norm{v}_{\h^k}} =0
\end{equation}
and we deduce that $\Phi$ is differentiable in $\mathcal{E}_K \cap \h^k$ with $D\Phi(u) \in \L(\h^k;\h^k)$ given by $D\Phi(u)v = Df(u) v$.  

Next, suppose that we have proved $\Phi$ is $\ell$ times differentiable for $1\le \ell \le m-1$ with derivatives given by \eqref{hk_comp_Cm_0} for $1\le j \le \ell$.    Again fix $u \in B(u,r) \subseteq \mathcal{E}_K \cap \h^k$.  For $v \in B(0,r) \backslash \{0\}$ we then compute 
\begin{multline}
(D^\ell \Phi(u+v)  - D^\ell\Phi(u) - D^{\ell+1} f(u)v) (w_1,\dotsc,w_\ell) = (D^\ell f(u+v) -D^\ell f(u) - D^{\ell+1}f(u)v) \prod_{i=1}^\ell w_i  \\
= \int_0^1 (D^{\ell+1} f(u+tv) - D^{\ell+1}f(u))v \prod_{i=1}^\ell w_i dt = \int_0^1 \int_0^1 t D^{\ell+2} f(u+stv) v^2 \prod_{i=1}^\ell w_i ds dt. 
\end{multline}
Again using Theorem \ref{hk_composition} and the fact that $\h^k$ is an algebra, we then have that 
\begin{multline}
\norm{(D^\ell \Phi(u+v)  - D^\ell\Phi(u) - D^{\ell+1} f(u)v) (w_1,\dotsc,w_\ell)}_{\h^k} \le \int_0^1 \int_0^1 t \norm{D^{\ell+2} f(u+stv) v^2 \prod_{i=1}^\ell w_i}_{\h^k} ds dt \\
\ls \norm{v}_{\h^k}^2 \prod_{i=1}^\ell \norm{w_i}_{\h^k} \int_0^1 \int_0^1 t \norm{f}_{C^{k+\ell+2}_b(K)}(1+ \norm{u+st v}_{\h^k} )^k dsdt  \\
\ls \norm{v}_{\h^k}^2 \prod_{i=1}^\ell \norm{w_i}_{\h^k}  \norm{f}_{C^{k+\ell+2}_b(K)} (1+ \norm{u}_{\h^k} + r)^k.
\end{multline}
Thus, 
\begin{equation}
 \norm{D^\ell \Phi(u+v)  - D^\ell\Phi(u) - D^{\ell+1} f(u)v }_{\L^\ell } \ls \norm{v}_{\h^k}^2   \norm{f}_{C^{k+\ell+2}_b(K)} (1+ \norm{u}_{\h^k} + r)^k,
\end{equation}
and we deduce from this that $D^\ell \Phi$ is differentiable at $u$ with 
\begin{equation}
D^{\ell+1} \Phi(u)(w_1,\dotsc,w_{\ell+1}) = D^{\ell+1}f(u) \prod_{i=1}^{\ell+1} w_i. 
\end{equation}
Thus, $\Phi$ is $(\ell+1)$-times differentiable and \eqref{hk_comp_Cm_0} holds for $1\le j \le \ell+1$.

Proceeding by finite induction, we deduce that $\Phi$ is $m$-times differentiable in $\h^k \cap \mathcal{E}_K$, so to show that it is $C^m$ it remains only to prove that $D^m\Phi$ is continuous.  Once more fix $u \in B(u,r) \subseteq \mathcal{E}_K \cap \h^k$.  For $v \in B(0,r) \backslash \{0\}$ we compute 
\begin{equation}
(D^m \Phi(u+v)  - D^m\Phi(u)) (w_1,\dotsc,w_m) = (D^m f(u+v) -D^m f(u)) \prod_{i=1}^m w_i  
= \int_0^1 D^{m+1} f(u+tv)v \prod_{i=1}^m w_i dt 
\end{equation}
and then estimate 
\begin{equation}
 \norm{(D^m \Phi(u+v)  - D^m\Phi(u)) (w_1,\dotsc,w_m)}_{\h^k} \ls \norm{v}_{\h^k} \prod_{i=1}^m \norm{w_i}_{\h^k} \norm{f}_{C^{k+m+1}_b(K)} (1+ \norm{u}_{\h^k} + r)^k.
\end{equation}
Thus, 
\begin{equation}
  \norm{D^m \Phi(u+v)  - D^m\Phi(u) }_{\L^m} \ls \norm{v}_{\h^k} \norm{f}_{C^{k+m+1}_b(K)} (1+ \norm{u}_{\h^k} + r)^k
\end{equation}
and we deduce that $D^m \Phi$ is continuous at $u$.  Hence, $\Phi : \h^k \cap \mathcal{E}_K \to \h^k$ is indeed $C^m$ for all $K \in \mathcal{K}(U)$.
\end{proof}

\section{Bubble construction }
 
In this section we carry out the proof of Theorem \ref{main_thm}.  First, we reformulate the third equation in \eqref{hydrostatic_eqns} in a form amenable to study in the weighted Sobolev functional framework developed in the previous section.  Then we study the smoothness of certain maps and employ a bifurcation argument to prove the theorem.

\subsection{Reformulation}
 
We aim to solve \eqref{hydrostatic_eqns} with hydrostatic bubble domains of the form $\Omega_{\m{int}}=\Lambda(\mathbb{B}^3)$, where $\mathbb{B}^3 = B(0,1) \subset \R^3$ is the usual open unit ball and $\Lambda: \overline{\mathbb B^3} \to \R^3 $ 
is a smooth diffeomorphism onto its image of the form $\Lambda(x) = \lambda(x_3) x$ for a smooth function $\lambda : [-1,1] \to \R_+$. 
This means that the bubble's unknown surface is $\Sigma = \Lambda(\mathbb{S}^2)$, the image of the unit sphere $\mathbb{S}^2 = \p \mathbb{B}^3$. Note that the spherical bubbles constructed in Section \ref{subsec:1.2} without gravity are realized by taking $\lambda = R_\alpha$ where $R_\alpha$ is given in \eqref{Ralpha_def}. In this case $\Lambda$ is a simple dilation map and the bubble surface is the image of the unit sphere, the sphere with radius $R_\alpha$.  We next give a sufficient condition on $\lambda$ for $\Lambda$ to be an injection in a slab domain $\R^{n-1} \times [-1,1]$, which contains the closed unit ball $\overline{\mathbb B^n} \subset \R^n$.

\begin{lem}\label{lem:injective}
Suppose that $\lambda : [-1,1] \to \R_+$ is continuous and such that $[-1,1] \ni t \mapsto t \lambda(t) \in \R$ is increasing.  Then for $2 \le n \in \N$ the map $\Lambda : \R^{n-1} \times [-1,1] \to \R^n$ given by $\Lambda(x) = \lambda(x_n) x$ is injective.
\end{lem}
\begin{proof}
Write $E = \R^{n-1} \times [-1,1]$ and denote a point $x \in E$ as $x = (x',x_n)$.  Suppose that $x,y \in E$ with $x \neq y$.  We break to two cases.  In the first $x_n \neq y_n$. Then the increasing assumption implies that $x_n \lambda(x_n) \neq y_n \lambda(y_n)$, and hence  $\Lambda(x) \neq \Lambda(y)$.  In the second $x_n = y_n$, so $x' \neq y'$.  Then $\Lambda'(x) = \lambda(x_n) x' = \lambda(y_n) x' \neq \lambda(y_n) y' = \Lambda'(y)$ and hence  $\Lambda(x) \neq \Lambda(y)$.
Thus, in any case $\Lambda(x) \neq \Lambda(y)$, and we deduce that $\Lambda$ is injective.
\end{proof}

Under the assumptions on $\lambda$ in Lemma \ref{lem:injective}, the restriction of $\Lambda$ to $\overline{\mathbb B^3}$ is a homeomorphism onto its image. Furthermore, if $\lambda$ is smooth  and sufficiently close to a constant in the $C^1_b$ norm, then $\Lambda: \overline{\mathbb B^3} \to \Lambda(\overline{\mathbb{B}^3})$ is a smooth diffeomorphism onto its image.  Our goal is thus to verify these conditions in proving Theorem \ref{main_thm}.

We now turn to a discussion of the differential equation \eqref{hydrostatic_eqns}. 
Suppose that $\Sigma = \Lambda(\mathbb{S}^2)$ for $\Lambda : \overline{\mathbb{B}^3} \to \R^3$ a diffeomorphism onto its image given by  $\Lambda(x) = \lambda(x_3) x$  with $\lambda : [-1,1] \to \R_+$   a $C^2$ map.    Then for $x= (x_1,x_2,\zeta) \in \mathbb{S}^2$ the total curvature at $y=\Lambda(x) \in \Sigma$ is (writing $' = D = d/d\zeta$)
\begin{multline}
 \mathscr{K}_{\Sigma}(y) =    \frac{ \zeta \lambda'(\zeta) + \lambda(\zeta)}{\lambda(\zeta)  \sqrt{  (\lambda(\zeta))^2 + (\lambda'(\zeta))^2 (1-\zeta^2)  }}  \\
+ \frac{- (1-\zeta^2) \lambda(\zeta) \lambda''(\zeta) + 2(1-\zeta^2) (\lambda'(\zeta))^2 + \zeta \lambda(\zeta) \lambda'(\zeta) + (\lambda(\zeta))^2}{(  (\lambda(\zeta))^2 + (\lambda'(\zeta))^2 (1-\zeta^2)  )^{3/2}},
\end{multline}
as can be computed by introducing the cylindrical parameterization 
\begin{equation}
 (0,2\pi) \times (-1,1) \ni (\theta,\zeta) \mapsto  (\cos(\theta) \sqrt{1-\zeta^2}, \sin(\theta) \sqrt{1-\zeta^2}, \zeta) \in \mathbb{S}^2
\end{equation}
and computing with standard formulas for the second fundamental form (see, for instance, Chapter 2 of \cite{Spivak_1979}).  Thus, the third equation of \eqref{hydrostatic_eqns} can be reformulated as the differential equation (writing $\lambda = \lambda(\zeta)$ for brevity):  
\begin{multline}
 0 = \sigma \left[ \frac{\zeta \lambda' + \lambda}{\lambda  \sqrt{  \lambda^2 + \lambda'^2 (1-\zeta^2)  }}  
+ \frac{- (1-\zeta^2) \lambda \lambda'' + 2(1-\zeta^2) \lambda'^2 + \zeta \lambda \lambda' + \lambda^2}{(  \lambda^2 + \lambda'^2 (1-\zeta^2)  )^{3/2}} \right] \\
+ (\pext^\ast- \rho_{\m{ext}} g \lambda \zeta) 
- \pfint(\alpha-g \lambda \zeta), 
\end{multline}
or equivalently (multiplying by the denominators and performing some simple algebra)
\begin{multline}\label{curvature_equation_lambda}
0 =     -  \sigma [ ( (1-\zeta^2) \lambda')' - 2\lambda    ] 
+\sigma \lambda^{-2}  \left[    \zeta (1-\zeta^2) (\lambda')^3 + 3(1-\zeta^2) \lambda(\lambda')^2     \right] \\
+ \lambda^{-1}(  \lambda^2 + \lambda'^2 (1-\zeta^2)  )^{3/2} \left[ (\pext^\ast- \rho_{\m{ext}} g \lambda \zeta) -\pfint (\alpha-g \lambda \zeta)   
\right] 
\end{multline}
for $\zeta \in [-1,1]$.
 
Synthesizing the above, we arrive at our strategy for proving Theorem \ref{main_thm}.  We will construct $\lambda \in C^\infty([-1,1])$ solving \eqref{curvature_equation_lambda} and of the form $\lambda = R + \varphi$ for $R \in \R_+$ a constant and $\varphi \in C^\infty([-1,1])$ with $\norm{\varphi}_{C^1_b}$ small enough that the induced map  $\Lambda: \overline{\mathbb B^3} \to \Lambda(\overline{\mathbb{B}^3})$  is a smooth diffeomorphism.  We will then define $\Sigma$ and $\Omega_{\m{int}}$ as stated in terms of $\Lambda$.

\subsection{Proof of main result}

We will solve \eqref{curvature_equation_lambda} by way of a bifurcation argument in $\h^k$ for appropriate $k \in \N$.  To this end, we must first establish that various maps are well-defined and smooth.  We begin this task with a quartet of lemmas.

\begin{lem}\label{map_lemma_1}
Let $\mathcal{W}_0 = \h^2[\R_+]$ be as defined in Theorem \ref{hk_comp_Cm}, and let $s \in \R$.  Then the map 
\begin{equation}
 \h^k \cap \mathcal{W}_0 \ni w \mapsto w^{s} \in \h^k 
\end{equation}
is smooth  for all $2 \le k \in \N$.  Moreover, $\lambda \in \mathcal{W}_0$ for every constant function $\lambda >0$.  
\end{lem}
\begin{proof}
 This is an immediate consequence of Theorem \ref{hk_comp_Cm} since $\R_+ \ni z \mapsto z^s \in \R_+$ is smooth.
\end{proof}

\begin{lem}\label{map_lemma_2}
Suppose that $3 \le k \in \N$.  Then the following hold.
\begin{enumerate}
 \item Let $p_0: \R^2 \to \R$ and $p_1 : \R^3 \to \R$ be polynomials.  Then the map 
\begin{equation}
 \h^{k+1} \ni u \mapsto   p_0(\zeta,u) + (1-\zeta^2) p_1(\zeta,u,Du)  \in \h^k
\end{equation}
is well-defined and smooth.  

\item The pre-image set $\mathcal{W}_1 \subseteq \h^{4}$ defined via
\begin{equation}
 u \in \mathcal{W}_1  \Leftrightarrow  u^2+ (1-\zeta^2)(Du)^2 \in \mathcal{W}_0,
\end{equation}
for $\mathcal{W}_0 \subseteq \h^2$ the open set from Lemma \ref{map_lemma_1}, is well-defined and open.  Moreover,  $\lambda  \in \mathcal{W}_1$ for every constant function $\lambda \in \R\backslash \{0\}$, and  the map 
\begin{equation}
 \mathcal{W}_1 \cap \h^{k+1} \ni u \mapsto (u^2+ (1-\zeta^2)(Du)^2)^{3/2} \in \h^k
\end{equation}
is well-defined and smooth.
\end{enumerate}
\end{lem}
\begin{proof}
We know from Proposition \ref{hkhm_bound_spec} that  if $u \in \h^{k+1}$ for $k \ge 1$, then $u \in \h^k$ and $Du \in \h^{k-1}$, and we also know that $Lu \in \h^{k-1}$ and  $\mathcal{P} \subseteq \h^k$.  Combining these observations with Theorem \ref{hs_algebra} shows that $\h^{k+1} \ni u \mapsto   p_0(\zeta,u) \in \h^k$ is well-defined and smooth. Similarly, since $k-1 \ge 2$, Theorem \ref{hs_algebra}  shows that the map $\h^{k+1} \ni u \mapsto q(\zeta,u, Du, Lu) \in \h^{k-1}$ is well-defined and smooth for any polynomial $q : \R^4 \to \R$.  We will use the latter fact to handle $p_1$.   

Let $f[u]:(-1,1) \to \R$ be given by $f[u](\zeta) = (1-\zeta^2) p_1(\zeta,u(\zeta),Du(\zeta))$.  Using the definition of $L$, we then compute  
\begin{multline}
 Df[u](\zeta) =  -2\zeta p_1(\zeta,u(\zeta), Du(\zeta)) + (1-\zeta^2)(\p_1 p_1(\zeta, u(\zeta), Du(\zeta)) + \p_2 p_1(\zeta, u(\zeta), Du(\zeta)) Du(\zeta)) \\
 + \p_3 p_1(\zeta, u(\zeta),Du(\zeta))(2\zeta Du(\zeta) - Lu(\zeta)) =: q(\zeta,u(\zeta), Du(\zeta), Lu(\zeta))
\end{multline}
for a polynomial $q : \R^4 \to \R$.  Thus, the above argument shows that $\h^{k+1} \ni u \mapsto Df[u] \in \h^{k-1}$ is well-defined and smooth.  On the other hand, Theorem \ref{hk_weight_improvement} shows that the map $\h^{k+1} \ni u \mapsto f[u](0) \in \R$ is well-defined and smooth.  Combining these observations with the identity $(1-\zeta^2) p_1(\zeta,u,Du) = f[u](0) + \mathcal{I} Df[u]$, where $\mathcal{I} : \h^{k-1} \to \h^k$ is the bounded linear map from the fifth item of Theorem \ref{hk_basics}, then shows that $\h^{k+1} \ni u \mapsto   (1-\zeta^2) p_1(\zeta,u, Du) \in \h^k$ is well-defined and smooth, which completes the proof of the first item.

The first item then shows that the map $\h^{k+1} \ni u \mapsto   u^2+ (1-\zeta^2)(Du)^2 \in \h^k$
is well-defined and smooth, which then shows that $\mathcal{W}_1 \subseteq \h^4$ is a well-defined open set containing all nonzero constant functions.  The second assertion of the second item then follows by composing the previous map  with the map from Lemma \ref{map_lemma_1} when $s =3/2$.
\end{proof}

\begin{lem}\label{map_lemma_3}
Suppose that $2 \le k \in \N$.  Then the map 
\begin{equation}
 \R \times \h^{k} \ni (g,u) \mapsto \pext^\ast - \rho_{\m{ext}} g \zeta u \in \h^{k}
\end{equation}
is well-defined and smooth.
\end{lem}
\begin{proof}
This follows from  Theorem \ref{hs_algebra} and the observation that the map $\R \ni t \mapsto t  \in \h^k$ (in the latter $t\in \h^k$ is the constant function with value $t$) is smooth.
\end{proof}

\begin{lem}\label{map_lemma_4}
Let $2 \le k \in \N$.  Then the following hold.
\begin{enumerate}
 \item  Let $\mathcal{W}_2 = \h^2[(\hint^{\m{min}}, \hint^{\m{max}})]$ be as defined by Theorem \ref{hk_comp_Cm}.  Then the map 
\begin{equation}
 \h^k \cap  \mathcal{W}_2 \ni w \mapsto \pfint(w) \in \h^k
\end{equation}
is well-defined and smooth.
 \item The map $\Theta : (\hint^{\m{min}}, \hint^{\m{max}}) \times \R \times \h^2 \to \h^2$ given by $\Theta(\alpha,g,u) = \alpha - g \zeta u$ is well-defined and smooth.  In particular, the pre-image open set $\mathcal{O} = \Theta^{-1}(\mathcal{W}_2) \subseteq (\hint^{\m{min}}, \hint^{\m{max}}) \times \R \times \h^2$ is well-defined and open.
  \item Define the open subset $\mathcal{O}^k \subseteq (\hint^{\m{min}}, \hint^{\m{max}}) \times \R \times \h^k$ via $\mathcal{O}^k = [(\hint^{\m{min}}, \hint^{\m{max}}) \times \R \times \h^k] \cap \mathcal{O}$.  Then 
\begin{equation}
 (\hint^{\m{min}}, \hint^{\m{max}}) \times \{0\} \times \h^{k} \subseteq \mathcal{O}^k,
\end{equation}
and the map 
\begin{equation}
 \mathcal{O}^k \ni (\alpha,g,u) \mapsto \pfint(\alpha - g\zeta u) \in \h^k 
\end{equation}
is well-defined and smooth.
\end{enumerate}
\end{lem}
\begin{proof}
The first item follows immediately from Theorem \ref{hk_comp_Cm}, and the second item follows from  Theorem \ref{hs_algebra} as in the proof of Lemma \ref{map_lemma_3}.  Then the third item follows by composing the maps from the first and second items.
\end{proof}

With Lemmas \ref{map_lemma_1}--\ref{map_lemma_4} in hand, we can now state a key result.

\begin{prop}\label{map_prop_1}
Suppose that $3 \le k \in \N$, let $\mathcal{W}_0 \subseteq \h^2$ be as in Lemma \ref{map_lemma_1}, $\mathcal{W}_1 \subseteq \h^4$ be as in Lemma \ref{map_lemma_2}, and $\mathcal{O} \subseteq  (\hint^{\m{min}}, \hint^{\m{max}}) \times \R \times \h^2$ be as in Lemma \ref{map_lemma_4}.  Define the open set 
\begin{equation}
 \mathcal{V}^{k+1} =   [(\hint^{\m{min}}, \hint^{\m{max}}) \times \R \times    (\h^{k+1} \cap \mathcal{W}_0  \cap \mathcal{W}_1   )   ] \cap \mathcal{O}  \subset (\hint^{\m{min}}, \hint^{\m{max}})\times \R \times \h^{k+1}.
\end{equation}
Then $(\hint^{\m{min}}, \hint^{\m{max}}) \times \{0\} \times (\h^{k+1} \cap \mathcal{W}_0  \cap \mathcal{W}_1   ) \subseteq   \mathcal{V}^{k+1}$, and the mapping $\Phi : \mathcal{V}^{k+1} \to \h^k$ defined by 
\begin{multline}
 \Phi(\alpha,g,u) = 
\sigma u^{-2}  \left[   \zeta (1-\zeta^2) (Du)^3 + 3(1-\zeta^2) u(Du)^2     \right] \\
+ u^{-1} (  u^2 + (Du)^2 (1-\zeta^2)  )^{3/2} \left[ (\pext^\ast- \rho_{\m{ext}} g u \zeta) -\pfint (\alpha-g u \zeta)   
\right] 
\end{multline}
is well-defined and smooth. 
\end{prop}
\begin{proof}
Lemma \ref{map_lemma_1} shows that the maps $\h^{k+1} \cap \mathcal{W}_0 \ni u \mapsto u^{-1} \in \h^{k}$ and $\h^{k+1} \cap \mathcal{W}_0 \ni u \mapsto u^{-2} \in \h^{k}$ are well-defined and smooth.  Lemma \ref{map_lemma_2} shows that the maps $\h^{k+1} \ni u \mapsto   \zeta (1-\zeta^2) (Du)^3 + 3(1-\zeta^2) u(Du)^2   \in \h^{k}$ and $\h^{k+1} \cap \mathcal{W}_1 \ni u \mapsto (  u^2 + (Du)^2 (1-\zeta^2)  )^{3/2}\in \h^k$ are well-defined and smooth.  Lemma \ref{map_lemma_3} shows that the map $\R \times \h^{k+1} \ni (g,u) \mapsto \pext^\ast- \rho_{\m{ext}} g u \zeta \in \h^k$ is well-defined and smooth.  Lemma \ref{map_lemma_4} shows that the map $\mathcal{V}^{k+1} \ni (\alpha,g,u) \mapsto \pfint (\alpha-g u \zeta)  \in \h^k$ is well-defined and smooth.  Combining these observations with Theorem \ref{hs_algebra} then shows that $\Phi$ is well-defined and smooth.  The stated inclusion in $\mathcal{V}^{k+1}$ follows from the properties of $\mathcal{W}_0$, $\mathcal{W}_1$, and $\mathcal{O}^{k+1}$ established in Lemmas \ref{map_lemma_1}, \ref{map_lemma_2}, and \ref{map_lemma_4}, respectively.
\end{proof}

We next record two more important mapping results.  The first is a variant of Proposition \ref{map_prop_1}.

\begin{prop}\label{map_prop_2}
Assume the hypotheses of Proposition \ref{map_prop_1}.  Suppose that \eqref{necessary_condition} holds and define the open interval $\mathcal{I} =  (\pfint^{-1}( 2\sigma / R^{\m{max}} + \pext^\ast), \hint^{\m{max}}) \subseteq \R$.  Then the pre-image set $\mathcal{U}^{k+1} \subseteq \mathcal{I} \times \R \times \h^{k+1}$ given by  
\begin{equation}
 (\alpha,g,u) \in  \mathcal{U}^{k+1} \Leftrightarrow (\alpha,g, R_\alpha + u) \in \mathcal{V}^{k+1},
\end{equation}
where $R_\alpha$ is defined by \eqref{Ralpha_def}, is well-defined and open.  Moreover, $\mathcal{I} \times \{0\} \times \{0\} \subseteq \mathcal{U}^{k+1}$, the map $\Psi : \mathcal{U}^{k+1} \to \h^k$ defined by $\Psi(\alpha,g,u) = \Phi(\alpha,g,R_\alpha + u) \in \h^k$ is well-defined and smooth, and  $\Psi(\alpha,0,0)=- 2\sigma R_\alpha$  
for all $\alpha \in \mathcal{I}$.
\end{prop}
\begin{proof}
The map  
\begin{equation}
 \mathcal{I} \times \R \times \h^{k+1} \ni (\alpha,g,u) \mapsto (\alpha,g, R_\alpha + u) \in  \mathcal{I} \times \R \times \h^{k+1}
\end{equation}
is well-defined and smooth, so the pre-image set $\mathcal{U}^{k+1}$ is well-defined and open.  Then $\Psi$ is well-defined and smooth thanks to these observations and Proposition \ref{map_prop_1}.  The containment assertion follows from the properties of $\mathcal{V}^{k+1}$ shown in Proposition \ref{map_prop_1}, and the identity  $\Psi(\alpha,0,0)=- 2\sigma R_\alpha$ follows by direct computation. 
\end{proof}

Our final mapping result handles the Legendre differential operator.

\begin{prop}\label{map_prop_3}
Suppose that $2 \le k \in \N$.  Suppose that \eqref{necessary_condition} holds, and define $\mathcal{I}$ as in Proposition \ref{map_prop_2}.   Then the map 
\begin{equation}
 \mathcal{I} \times \h^{k+2} \ni (\alpha,v) \mapsto -\sigma[ D((1-\zeta^2) D v) - 2(R_\alpha + v)   ]    \in \h^k
\end{equation}
is well-defined and smooth.
\end{prop}
\begin{proof}
The map in question is of the form $(\alpha,v) \mapsto \sigma[Lv + 2(R_\alpha + v)]$ for $L$ the Legendre differential operator.  As such, the well-definedness and smoothness of the map follow from Theorem \ref{hk_basics}.
\end{proof}

We can now state our main result concerning solutions to \eqref{curvature_equation_lambda}.

\begin{thm}\label{bifurcation_thm}
Suppose that \eqref{necessary_condition} holds, and define $\mathcal{I}$ as in Proposition \ref{map_prop_2}. There exist an open, symmetric interval $0 \in G \subseteq \R$ and maps $A : G \to \mathcal{I}$ and $\varphi : G \to C^\infty([-1,1])$ such that the following hold.
\begin{enumerate}
 \item $A(0)=0$, and $A$ is smooth.
 \item $\varphi(0) =0$, and for every $m \in \N$ the induced map $\varphi : G \to C^m_b([-1,1])$ is smooth.
 \item For $g \in G$ the equation \eqref{curvature_equation_lambda} is satisfied with $\alpha = A(g)$ and $\lambda = R_{A(g)} + \varphi(g)$. 

 \item Write $\h^4_\bot = \{u \in \h^4 \st \hat{u}(1) =0\}$ and $\mathcal{U}^4_\bot = \mathcal{U}^4 \cap ( \R \times \R \times \h^4_\bot)$.  There exists an open set $0 \in U \subseteq \mathcal{U}^4_\bot$ such that
\begin{equation}\label{E:unique}
\{  (\alpha,g,u) \in U  \st \sigma[Lu + 2(R_\alpha +u)] +\Psi(\alpha,g,u) =0 \} 
= \left[ \{ (\alpha,0,0): \alpha \in \mathcal I\} \cup \{ ( A(g), g, \varphi (g) ): g\in G \}\right] \cap U,
\end{equation}  
where $\Psi$ is as in Proposition \ref{map_prop_2}.

 \item For all $g \in G$ we have the reflection symmetries $A(-g) = A(g)$ and $\varphi(-g)(\zeta) = \varphi(g)(-\zeta)$. 
 
 \item As $g \to 0$ we have the asymptotic expansions 
\begin{equation}
 A(g) =  - g^2 \frac{\pfint'''(0)R_0^2 }{ 10\pfint''(0) } + O(g^4)
\text{ and }\varphi(g) = g^2 \frac{\pfint''(0) R_0^4}{6\sqrt{10}\,\sigma}(p_2 - \sqrt{5}\, p_0)  - g^3 \frac{\pfint'''(0) R_0^5}{75\sqrt{14}\,\sigma} p_3 + O(g^4),
\end{equation}
where $p_2(\zeta) - \sqrt{5}\, p_0(\zeta) = \frac{3\sqrt{10}}{4}(\zeta^2 -1)$ and $p_3(\zeta) = \frac{\sqrt{14}}{4}(5\zeta^3 - 3 \zeta)$.
\end{enumerate}
\end{thm}
\begin{proof}
To begin we use Propositions \ref{map_prop_1}--\ref{map_prop_3} to define the map  $\Xi : \mathcal{U}^4 \to \h^2$ by
\begin{equation}\label{def:Xi}
\Xi(\alpha,g,u) = \sigma[ -D((1-\zeta^2) D u) + 2(R_\alpha + u)   ] + \Psi(\alpha,g, u),
\end{equation}
which is well-defined and smooth thanks to these propositions.  Note that 
\begin{equation}\label{bifurcation_thm_1}
 \Xi(\alpha,0,0) = 2\sigma R_\alpha + \Phi(\alpha,0,R_\alpha) = 2 \sigma R_\alpha + R_\alpha^2 [\pext^\ast - \pfint(\alpha)  ] = 2\sigma R_\alpha - 2\sigma R_\alpha = 0.
\end{equation}
We view $\mathcal{U}^4 \subseteq \mathcal{I} \times (\R \times \h^{4})$ and think of the latter set as being the product of two factors, for which we will write $\p_1$ and $D_2$ for the derivatives.  We then compute 
\begin{equation}\label{bifurcation_thm_2}
 D_2\Xi(\alpha,0,0)(\dot{g},\dot{u}) = (\pfint'(\alpha)-\rho_{\m{ext}})R_\alpha^3 \zeta \dot{g}  + \sigma(   -D[(1-\zeta^2) D\dot{u} ]  +2 \dot{u} ) + 2 R_\alpha (\pext^\ast - \pfint(\alpha)) \dot{u},
\end{equation}
for $\alpha \in \mathcal{I}$ and $(\dot{g},\dot{u}) \in \R \times \h^{4}$.  Using the definition of $R_0$ from \eqref{Ralpha_def}  and the fact that $\pfint'(0) = \rho_{\m{ext}}$, we deduce from \eqref{bifurcation_thm_2} that
\begin{equation}\label{bifurcation_thm_3}
 D_2\Xi(0)(\dot{g},\dot{u}) =  \sigma(   -D[(1-\zeta^2) D\dot{u} ]  -2 \dot{u} ).
\end{equation}
For $(\dot{g},\dot{u})$ fixed we also learn that 
\begin{equation}\label{bifurcation_thm_4}
 \partial_1 D_2\Xi(0)(\dot{g},\dot{u}) = \pfint''(0) R_0^3 \zeta \dot{g}.
\end{equation}

Next, we recall that the Legendre polynomial $p_1$ is of the form  $p_1(\zeta) = \sqrt{3/2} \zeta$. This and \eqref{bifurcation_thm_3} then show that 
\begin{equation}\label{bifurcation_thm_7}
D_2 \Xi(0)(\dot{g},\dot{u}) =0 \Leftrightarrow \dot{g} \in \R \text{ and } \dot{u} \in \R p_1
\end{equation}
and 
\begin{equation}\label{bifurcation_thm_8}
 D_2 \Xi(0)(\dot{g},\dot{u}) = w \Leftrightarrow \hat{w}(1) =0 \text{ and }  \hat{\dot{u}}(n) = \frac{\hat{w}(n)}{\sigma(\lambda_n - \lambda_1)} \text{ for } n \in \N \backslash\{1\},
\end{equation}
where $\lambda_n = n(n+1)$ is the Legendre eigenvalue associated to $p_n$.   These calculations suggest that we define the closed codimension-one subspace 
\begin{equation}
 \h^k_\bot = \{u \in \h^k \st \hat{u}(1) =0\} \subset \h^k \text{ for all }k \in \N.
\end{equation}
Then we set 
\begin{equation}
 \mathcal{U}^4_\bot = \mathcal{U}^4 \cap ( \R \times \R \times \h^4_\bot),
\end{equation}
which allows us to view the restriction $\Xi : \mathcal{U}^4_\bot \to \h^2$ as a smooth map.  In light of \eqref{bifurcation_thm_7} and \eqref{bifurcation_thm_8}, we now know that 
\begin{equation}
 D_2\Xi(0) (\dot{g},\dot{u}) = 0 \text{ for } (\dot{g},\dot{u}) \in \R \times \h^4_\bot \Leftrightarrow  (\dot{g},\dot{u}) \in \R \times \{0\}
\end{equation}
and 
\begin{equation}
D_2\Xi(0) (\dot{g},\dot{u}) = w  \text{ for } (\dot{g},\dot{u}) \in \R \times \h^4_\bot \text{ and }w \in \h^2 \Leftrightarrow \hat{w}(1) =0 \text{ and } \hat{w}(n) = \sigma(\lambda_n - \lambda_1) \hat{\dot{u}}(n)  \text{ for } n \in \N \backslash\{1\}.
\end{equation}
In other words, 
\begin{equation}\label{bifurcation_thm_9}
 \ker D_2 \Xi(0) = \spn\{(1,0)\} \subset  \R \times \h^4_\bot \text{ and } \ran D_2 \Xi(0) = \h^2_\bot.
\end{equation}
Moreover, \eqref{bifurcation_thm_4} and \eqref{bifurcation_thm_6} show that the Crandall-Rabinowitz transversality condition is satisfied:   
\begin{equation}\label{bifurcation_thm_10}
 \p_1 D_2 \Xi(0) (1,0) = \pfint''(0)R_0^3 \zeta \in \spn\{ p_1\} = (\h^2_\bot)^\bot.
\end{equation}

We are now in a position to apply the Crandall-Rabinowitz bifurcation theorem, which was originally proved in \cite{CR_1971} in a $C^2$ setting.  However, in order to take advantage of our $C^\infty$ structure we will use a form of the theorem recorded in Theorem 3.1 of \cite{Stevenson_2025}.  Applying this provides the open interval $0 \in G \subseteq \R$ and smooth maps $A : G \to (-r,r)\subset \mathcal{I}$  for some $r >0$ and $W: G \to  ( \ker D_2 \Xi(0))^\bot = \{0\} \times \h^4_\bot$ such that $W_2(0) =0$ and
\begin{equation}
\begin{split}
 \Xi(A(g), g( (1,0) + W(g)   )) &= \Xi(A(g),g + g W_1(g) , g W_2(g))=\Xi(A(g), g , g W_2(g))=0 \text{ for } g \in G.
\end{split}
\end{equation}
We then set 
\begin{equation}
  \phi(g) = R_{A(g)} + g W_2(g) 
\end{equation} 
to learn that 
\begin{equation}
 \sigma( - D[(1-\zeta^2) D \phi(g)  ]  + 2\phi(g) ) =  - \Phi(A(g),g,\phi(g)) \in \h^{3} \text{ for }g \in G. 
\end{equation}
Now, the key point is that if we define the linear operator $M = \sigma (L+2I)$ then $M : \h^{k+2} \to \h^k$ is an isomorphism for all $k \in \N$ thanks to Theorem \ref{hk_basics}.  Hence, we learn that 
\begin{equation}\label{bifurcation_thm_11}
 \phi(g) = -M^{-1} \Phi(A(g),g,\phi(g)) \in \h^{5}, 
\end{equation}
and we have elliptic promotion with $G \ni g \mapsto \phi(g) \in \h^5$ being smooth.  We may then bootstrap to see that 
\begin{multline}
 G \ni g \mapsto \phi(g) \in \h^{k+2} \text{ smooth implies that }  G \ni g \mapsto -M^{-1} \Phi(A(g),g,\phi(g)) \in \h^{k+3} \text{ is smooth, } \\
\text{ and hence } G \ni g \mapsto \phi(g) \in \h^{k+3} \text{ is smooth}.
\end{multline}
Thus $\phi(g) \in \bigcap_{k=0}^\infty \h^k = C^\infty([-1,1])$ for all $g \in G$, and the map 
$G \ni g \mapsto \phi (g) \in C^m_b([-1,1])$  is smooth for any fixed $m \in \N$.  The first two items are now proved upon setting $\varphi (g)= \phi (g)- R_{A(g)}$, and the third then follows by unpacking the meaning of $M$ and $\Phi$ in the identity \eqref{bifurcation_thm_11}.   The fourth item directly follows from Theorem 3.1 of \cite{Stevenson_2025}. In particular, we have  that: (i) if $\alpha \in (-r,r)$ and $ w \in  X \cap B_{X}(0,r)$ where $X=(\ker D_2 \Xi(0))^\bot$ satisfy $\Xi ( \alpha, g ((1,0)+ w) ) =0 $ for some $0\neq g\in G$,  then $\alpha=A(g)$ and $ w=W(g)$; and (ii) there exist an open neighborhood $U$ of $(0,0,0)$ and $c\in \mathbb R_+$ such that for all $g\in G$ and $(\alpha, z) \in \mathbb R \times X$ satisfying $\Xi( \alpha, g(1,0) + z)=0$, $(  \alpha, g(1,0) + z) \in U$ it holds that $|g \alpha| + \|z\|_X \le c g^2$ and $|g|\le r/(2c)$. Therefore, we deduce \eqref{E:unique}.  

We now turn to the proof of the fifth item. Shrinking $G$ if necessary, we may assume that $G$ is symmetric and that $(A(g),-g,\mathcal{R}\varphi(g)) \in U$ for all $g \in G$; replacing $U$ by $U \cap (\R \times G \times \h^4_\bot)$ keeps \eqref{E:unique} valid, so the first four items still hold.  Let $\mathcal{R} : \h^k \to \h^k$ denote the reflection operator $\mathcal{R}u(\zeta) = u(-\zeta)$.  It is then a simple matter to check that $L \mathcal{R} = \mathcal{R} L$ and  that the Legendre polynomials satisfy $p_n(-\zeta) = (-1)^n p_n(\zeta)$, which in particular implies that $\mathcal{R}\h^4_\bot = \h^4_\bot$.  An elementary but lengthier argument also shows that $\Xi(\alpha, -g, \mathcal{R}u) = \mathcal{R}\Xi(\alpha, g, u)$ for all $(\alpha,g,u) \in \mathcal{U}^4_\bot$.  Then for $g \in G$ we can compute $\Xi(A(g), -g, \mathcal{R}\varphi(g)) = \mathcal{R} \Xi(A(g), g, \varphi(g)) = \mathcal{R}0=0$.  On  the other hand, $\Xi(A(-g),-g,\varphi(-g)) =0$ for $g \in G$,  so the fourth item implies that $A(g) = A(-g)$ and $\varphi(-g) = \mathcal{R} \varphi(g)$.  This completes the proof of the fifth item.

Finally, we prove the sixth item.  We know from the first item that $A(0) = 0$, but since $A$ is even in $g$ by the fifth item, we can Taylor expand $A(g) = g^2 A_2 + O(g^4)$ for some $A_2 \in \R$. We saw above that $\varphi(g) = g W_2(g)$ with $W_2(0) =0$, so $\varphi(0) = \varphi'(0) =0$.  This allows us to Taylor expand   $\varphi(g) = g^2 \varphi_2 + g^3 \varphi_3 + O(g^4)$ for some $\varphi_2,\varphi_3 \in \h^4_\bot$, where $O(g^4)$ is relative to the $C^m_b([-1,1])$ norm for any fixed $m \ge 2$ (to be concrete we set $m=4$ in what follows).  The fifth item tells us that $\varphi_2 = \mathcal{R} \varphi_2$ and $\varphi_3 = - \mathcal{R} \varphi_3$.  

In fact, we can explicitly evaluate $A_2$, $\varphi_2$, and $\varphi_3$ by Taylor expanding  $\Xi(A(g), g, \varphi(g)) = 0$ in powers of $g$.  To do so we first rewrite this as 
\begin{equation}
 0 = \sigma (L +2 I) \lambda(g) + \Phi(A(g), g, \lambda(g)) \text{ for } \lambda(g) = R_{A(g)} + \varphi(g).
\end{equation}
Recalling that
\begin{equation}
R_0 = \frac{2\sigma}{\pfint(0) - \pext^\ast} \text{ and }  \left. \frac{d R_\alpha}{d\alpha} \right|_{\alpha=0} = -\frac{2\sigma \pfint'(0)}{(\pfint(0) - \pext^\ast)^2} = -\frac{R_0^2 \pfint'(0)}{2\sigma},
\end{equation}
we may compute 
\begin{equation}
 R_{A(g)} = R_0 - g^2 \frac{R_0^2 \rho_{\m{ext}} A_2}{2\sigma} + O(g^4) \text{ and hence } \lambda(g) = R_0 + g^2\left( \varphi_2 - \frac{R_0^2 \rho_{\m{ext}} A_2}{2\sigma} \right) + g^3\varphi_3 + O(g^4).
\end{equation}
Next, we expand the nonlinear operator $\Phi(A(g), g, \lambda)$ defined in Proposition \ref{map_prop_1} using $D\lambda(g) = O(g^2)$; this yields
\begin{equation} 
\Phi(A(g), g, \lambda(g)) = \lambda(g)^2 \left[ \pext^\ast - \rho_{\m{ext}} g \lambda(g) \zeta - \pfint(A(g) - g \lambda(g) \zeta) \right] + O(g^4).
\end{equation}
On the one hand, we may expand 
\begin{equation}
 \lambda(g)^2 = R_0^2 + 2R_0 g^2\left( \varphi_2 - \frac{R_0^2 \rho_{\m{ext}} A_2}{2\sigma} \right) + 2R_0 g^3 \varphi_3 + O(g^4).
\end{equation}
On the other, since $\pext^\ast - \pfint(0) = -2\sigma/R_0$ and $\pfint'(0) = \rho_{\m{ext}}$,
\begin{multline}
\pext^\ast - \rho_{\m{ext}}g\lambda(g) \zeta   -\pfint(0) - \pfint'(0)(A_2 g^2 - g\lambda(g)\zeta) \\
- \hal\pfint''(0)(A_2 g^2 - g\lambda(g)\zeta)^2 
- \frac{1}{6}\pfint'''(0)(-g\lambda(g)\zeta)^3 + O(g^4)  \\
= -\frac{2\sigma}{R_0} - g^2 \left[ \rho_{\m{ext}} A_2 + \hal \pfint''(0) R_0^2 \zeta^2 \right] + g^3 \left[ \pfint''(0) R_0 A_2 \zeta + \frac{1}{6} \pfint'''(0) R_0^3 \zeta^3 \right] + O(g^4).
\end{multline}
Combining these then shows that 
\begin{multline}
 \Phi(A(g), g, \lambda(g)) = -2 \sigma R_0 + g^2 \left[-4 \sigma \varphi_2 + R_0^2 \rho_{\m{ext}} A_2
   - \hal \pfint''(0) R_0^4 \zeta^2 \right] \\
+ g^3 [-4\sigma \varphi_3 + R_0^2 \pfint''(0) R_0 A_2 \zeta + \frac{1}{6} \pfint'''(0) R_0^5 \zeta^3] 
+ O(g^4).
\end{multline}
Similarly, we expand
\begin{equation} 
\sigma[L\varphi(g) + 2\lambda(g)] = 2\sigma R_0 + g^2 \left[ \sigma (L +2I) \varphi_2  - R_0^2 \rho_{\m{ext}} A_2  \right]  + g^3 [ \sigma (L+2I) \varphi_3 ] + O(g^4). 
\end{equation}
We then plug these expansions into $\Xi(A(g),g, \varphi(g))=0$ and collect in powers of $g$.  All terms at orders zero and one cancel; at order $g^2$ we are left with 
\begin{equation}
\sigma(L-2I)\varphi_2 = \hal \pfint''(0) R_0^4 \zeta^2 = \frac{1}{8} \pfint''(0) R_0^4 (L-2I) (\zeta^2-1),
\end{equation}
and hence 
\begin{equation}
\varphi_2(\zeta) = \frac{\pfint''(0) R_0^4}{8\sigma} (\zeta^2 - 1) = \frac{\pfint''(0) R_0^4}{6\sqrt{10}\,\sigma} (p_2(\zeta) - \sqrt{5}\, p_0(\zeta)).
\end{equation}
At order $g^3$ we have 
\begin{multline}
\sigma(L-2I)\varphi_3 = - \pfint''(0) R_0^3 A_2 \zeta - \frac{1}{6} \pfint'''(0) R_0^5 \zeta^3 \\
= \sqrt{\tfrac{2}{3}} \left( - \pfint''(0) R_0^3 A_2 - \frac{1}{10} \pfint'''(0) R_0^5 \right) p_1 - \frac{1}{15} \sqrt{\tfrac{2}{7}}\, \pfint'''(0) R_0^5 p_3,
\end{multline}
and since $\sigma(L-2I)\varphi_3 \bot p_1$ we deduce from this that 
\begin{equation}
 A_2 = -\frac{\pfint'''(0)R_0^2 }{ 10\pfint''(0) }
\text{ and }
\sigma(L-2I)\varphi_3 =- \frac{1}{15} \sqrt{\tfrac{2}{7}}\, \pfint'''(0) R_0^5 p_3,
\end{equation}
the latter of which then implies that 
\begin{equation}
\varphi_3(\zeta) = - \frac{\pfint'''(0) R_0^5}{75\sqrt{14}\,\sigma} p_3(\zeta) = - \frac{\pfint'''(0) R_0^5}{300\sigma} (5\zeta^3 - 3\zeta).
\end{equation} 
\end{proof}

Finally, we are in a position to present the proof of our main theorem.

\begin{proof}[Proof of Theorem \ref{main_thm}]
Let $G$, $A$, and $\varphi$ be as in Theorem \ref{bifurcation_thm}, and select $R_0 < r < R^{\m{max}}$, where the latter was defined in \eqref{R_max_def}.   Shrinking $G$ if necessary, we may assume the following three conditions hold.  First, $A(g) -g [-r,r] \subseteq \hint(\R_+)$ for all $g \in G$, i.e. \eqref{rhoint_def_compat} holds with $\alpha = A(g)$.  Second, $\norm{\varphi(g)}_{C^1_b}$ is sufficiently small for all $g \in G$ to guarantee that the map  $\Lambda_g : \overline{\mathbb B^3} \to \R^3$ defined by $\Lambda_g(x) = (R_{A(g)} +  \varphi(g)(x_3)) x$  is a smooth diffeomorphism onto $\Lambda_g ( \overline{\mathbb B^3})$ (see Lemma \ref{lem:injective} and the discussion after).  Third,  we have the bound $\norm{ R_{A(g)} + \varphi(g) }_{C^0_b} \le r$ for all $g \in G$, which guarantees that $\Lambda_g(\overline{\mathbb{B}^3}) \subseteq r \overline{\mathbb{B}^3} \subset \Omega_{\m{tot}}$.  We then define $\Omega_{\m{int},g}$, $\Omega_{\m{ext},g}$, $\Sigma_g$, and $\rho_{\m{int},g}$ as stated and deduce that all of the conclusions are satisfied thanks to Theorem \ref{bifurcation_thm}.
\end{proof}


\end{document}